\documentclass[11pt]{amsart}
\usepackage{amssymb,amsmath,color,hyperref}
\usepackage[mathscr]{eucal}

\theoremstyle{plain}
\newtheorem{thm}{Theorem}
\newtheorem{theorem}[thm]{Theorem}

\theoremstyle{definition}

\newtheorem{ex}[thm]{Example}

\theoremstyle{remark}
\newtheorem{rem}[thm]{Remark}

\newtheorem*{remark*}{Remark}

\numberwithin{equation}{section}

\newcommand{\cH}{{\mathcal{H}}}

\newcommand{\field}[1]{{\mathbb{#1}}}
\newcommand{\NN}{\field{N}}
\newcommand{\ZZ}{\field{Z}}

\newcommand{\RR}{\field{R}}
\newcommand{\CC}{\field{C}}

\allowdisplaybreaks

\begin{document}

\title[Trace formula for the magnetic Laplacian]{Trace formula for the magnetic Laplacian at zero energy level}



\author[Y. A. Kordyukov]{Yuri A. Kordyukov}
\address{Institute of Mathematics, Ufa Federal Research Centre, Russian Academy of Sciences, 112~Chernyshevsky str., 450008 Ufa, Russia} \email{yurikor@matem.anrb.ru}


\subjclass[2010]{Primary 58J37; Secondary 53D50}

\keywords{Magnetic Laplacian, trace formula, semiclassical asymptotics, Gutzwiller formula}


\begin{abstract}
The paper is devoted to the trace formula for the magnetic Laplacian associated with a magnetic system on a compact manifold. This formula is a natural generalization of the semiclassical Gutzwiller trace formula and reduces to it in the case when the magnetic field form is exact. It differs somewhat from the Guillemin-Uribe trace formula studied in the author's previous work with I.A. Taimanov. Moreover, in contrast to that work, the focus is on the trace formula at the zero energy level, which is a critical energy level. The paper gives an overview of the main notions and results related to the trace formula at the zero energy level, describes various approaches to its proof, and gives concrete examples of its computation. In addition, a brief review of the Gutzwiller trace formula for regular and critical energy levels is given.
\end{abstract}

\date{\today}

\dedicatory{Dedicated to I. A. Taimanov on the occasion of his 60th birthday}

\maketitle
\tableofcontents
\section{Introduction}
A magnetic system on a compact manifold $M$ of dimension $d$ is given by a Riemannian metric $g$ and a closed differential 2-form of the magnetic field $F$ on $M$. It is well known that if $F$ satisfies the quantization condition
\begin{equation}\label{e:prequantum}
[F]\in H^2(M,2\pi \mathbb Z),
\end{equation}
then there exists a Hermitian line bundle $(L,h^L)$ with a Hermitian connection $\nabla^L : C^\infty(M,L)\to C^\infty(M,T^*M\otimes L)$ such that the curvature form $R^L$ of the connection $\nabla^L$ is related to the form $F$ by the formula
\begin{equation}\label{e:1.2}
F = iR^L.
\end{equation}
The Riemannian metric on $M$ and the Hermitian structure on $L$ allow us to define inner products on $C^\infty(M,L)$ and $C^\infty(M,T^*M\otimes L)$ and the adjoint operator
$
(\nabla^L)^* : C^\infty(M,T^*M\otimes L)\to C^\infty(M,L).
$
The Bochner Laplacian associated with the bundle $L$ (or the magnetic Laplacian) is a second-order differential operator acting in the space $C^\infty(M,L)$ by the formula
\[
\Delta^{L} = (\nabla^L)^*\nabla^L.
\]

For any $N\in \NN$ consider the $N$th tensor power $L^N=L^{\otimes N}$ of the line bundle $L$. Denote by $\Delta^{L^N}$ the corresponding magnetic Laplacian in the space $C^\infty(M,L^N)$. The parameter $\hbar = \frac 1N$ can be treated as a semiclassical parameter and, accordingly, the passage to the limit $N\to \infty$ as a semiclassical limit. This point of view is well known and generally accepted in geometric quantization (see, for example, \cite{berezin74}). 

We will study the Bochner-Schr\"odinger operator $H_N$ acting on $C^\infty(M,L^N)$ by the formula
\begin{equation}\label{e:BochnerHN}
H_{N}=\Delta^{L^N}+NV,
\end{equation}
where $V\in C^\infty(M)$ is a real-valued function.
Additionally, one can consider an arbitrary Hermitian vector bundle $(E,h^E)$ with a Hermitian connection $\nabla^E$, but for the sake of simplicity, we will not consider this case in the paper.
Such an operator was introduced and studied by J.-P. Demailly in connection with holomorphic Morse inequalities for Dolbeault cohomology associated with high tensor powers of a holomorphic Hermitian bundle over a compact complex manifold \cite{Demailly85} (see also \cite{bismut87,Demailly91,ma-ma:book} and the references given there). This operator also has a quantum mechanical interpretation as a magnetic Schr\"odinger operator. It describes the motion of a charged quantum particle on a manifold $M$ in an external electromagnetic field given by the magnetic form $NF$ and the electric potential $NV$.  

If the Hermitian line bundle $(L,h^L)$ is trivial on an open subset $U$ of the manifold $M$, i.e.
\[
L\left|_U\right.\cong U\times \CC\ \text{and}\ |(x,z)|_{h^L}=|z|,\quad (x,z)\in U \times \CC,
\]
and the Hermitian connection $\nabla^L$ is written as
\begin{equation}\label{e:d-A}
\nabla^L = d-i \mathbf A : C^\infty(U)\to C^\infty(U,T^*U),
\end{equation}
with some real 1-form $\mathbf A$ (connection form or magnetic potential), then
\[
R^L=-id\mathbf A,\quad F=d\mathbf A.
\]
In this case, the operator $H_N$ has the form
\[
H_N=(d-iN\mathbf A)^*(d-iN\mathbf A)+NV, \quad N\in \NN.
\]
It is related to the semiclassical magnetic Schr\"odinger operator
\[
\mathcal H^\hbar=(i\hbar d+\mathbf A)^*(i\hbar d+\mathbf A)+\hbar V
\]
by the formula
\begin{equation}\label{e:Bochner-Hh}
H_N=\hbar^{-2}\mathcal H^\hbar, \quad \hbar=\frac{1}{N},\quad N\in \NN. 
\end{equation}

In particular, suppose that one can choose local coordinates $(x^1,\ldots,x^d)$ on $U$. We write the connection form as $\mathbf A= \sum_{j=1}^dA_j(x)\,dx^j$. Then $F$ has the form
\[
F=d\mathbf A=\sum_{j<k}F_{jk}\,dx^j\wedge dx^k, \quad
F_{jk}=\frac{\partial A_k}{\partial x^j}-\frac{\partial
A_j}{\partial x^k}.
\]
The operator $\Delta^{L^N}$ is written as
\begin{multline}\label{e:DLp}
H_N=-\frac{1}{\sqrt{|g(x)|}}\sum_{1\leq j,\ell\leq d}\left(\frac{\partial}{\partial x^j} -iNA_j(x)\right)\\ \times \left[\sqrt{|g(x)|}
g^{j\ell}(x) \left(\frac{\partial}{\partial
x^\ell}-iNA_\ell(x)\right)\right]+NV(x),
\end{multline}
where $g(x)=(g_{j\ell}(x))_{1\leq j,\ell\leq n}$ is the matrix of the Riemannian metric $g$, $g(x)^{- 1}=(g^{j\ell}(x))_{1\leq j,\ell\leq n}$ is its inverse matrix and $|g(x)|=\det(g( x))$.

If the form $F$ is exact on $M$, i.e., $F=d\mathbf A$ for some 1-form $\mathbf A$, then we say that the magnetic system is exact. In this case we will always assume that $(L,h^L)$ is the trivial Hermitian line bundle on $M$ and the Hermitian connection $\nabla^L$ is given by \eqref{e:d-A}.

The operator $H_N$ is a second-order self-adjoint elliptic differential operator on a compact manifold. Therefore, it has a discrete spectrum in $L^2(M,L^N)$ consisting of a countable set of eigenvalues of finite multiplicity. 

The corresponding classical dynamics is described by the magnetic geodesic flow on the cotangent bundle $T^*M$ (see section \ref{s:ham}). The studies of dynamic and variational problems for magnetic geodesic flows, started in the works of S.P. Novikov and I.A. Taimanov  \cite{Novikov1982,NT,T1,T2,T3}, have been actively continued in recent years. We are interested in the relationship of the asymptotic properties of the eigenvalues and eigenfunctions of the magnetic Laplacian $\Delta^{L^N}$ (and more generally the Bochner-Schr\"odinger operator $H_N$) as $N\to \infty$ with the dynamics of the magnetic geodesic flow. These questions were discussed in recent works of the author with I.A. Taimanov \cite{KT19,KT20,KT22}. One of the most important tools in such studies are trace formulas. In papers \cite{KT19,KT22} the Guillemin-Uribe trace formula for the magnetic Laplacian was studied. In this paper, we study the trace formula, which is somewhat different from the Guillemin-Uribe formula. This formula is a natural generalization of the semiclassical Gutzwiller trace formula and reduces to it in the case of an exact magnetic system. Moreover, unlike the papers \cite{KT19,KT22}, where the trace formulas were considered at regular energy levels, in this paper the focus is on the zero energy level, which is a critical energy level.  

The work is organized as follows.
We begin our presentation in Section \ref{s:G} with a brief overview of the semiclassical Gutzwiller trace formula, first in general form and then in the particular case of exact magnetic systems. In Section~\ref{s:trace0} we use the connection of exact magnetic systems with the semiclassical case to define the smoothed spectral density for the Bochner-Schr\"odinger operator associated with the general magnetic system. We then write down the trace formula at zero energy level, which is an asymptotic expansion for the smoothed spectral density, and sketch its proof using the methods of local index theory. We also present several other results concerning the asymptotic behavior of the low-lying eigenvalues of the Bochner-Schr\"odinger operator. Section~\ref{s:trace1} is devoted to approaches to proving the trace formula for the Bochner-Schr\"odinger operator by methods of microlocal analysis. Section \ref{s:examples} gives concrete examples of computing the trace formula at zero energy level for two-dimensional surfaces of constant curvature with constant magnetic fields, as well as for a constant magnetic field on a three-dimensional torus.

\section{Gutzwiller's trace formula}\label{s:G}
A semiclassical trace formula was proposed by Gutzwiller in \cite{Gu} (see also related papers by Balian and Bloch \cite{BB74}). The first rigorous mathematical proofs of this formula were given by Y. Colin de Verdi\`ere \cite{CdV73a, CdV73b}, J. Chazarain \cite{Cha74} and H. Duistermaat and V. Guillemin \cite{Du-Gu75} for the Laplace operator on a compact Riemannian manifold without boundary in the high-energy limit (the trace formula in this case is often called the Duistermaat-Guillemin formula). Here we should also mention the studies of the Selberg formula, begun in \cite{Selberg}. Rigorous proofs of Gutzwiller's formula for general semiclassical operators were first given for regular energy levels by R. Brummeluis and A. Uribe \cite{BU91}, E. Meinrenken \cite{M92,M94}, T. Paul and A. Uribe \cite {PU95} (see also more recent papers \cite{CRR99,PP98,CP99,Dozias,SZ02} and for critical energy levels R. Brummeluis, T. Paul and A. Uribe \cite{BPU95} and D. Khuat-Duy \cite{K-D97} (see also \cite{camus04,camus04a,camus04b,camus08}) We also mention reviews \cite{Uribe00,CdV07} and more recent works \cite{Savale18,SZ21}.

The Gutzwiller formula is an asymptotic (in the semiclassical limit) formula for the so-called smoothed spectral density, which describes the distribution of the eigenvalues of a quantum Hamiltonian in some neighborhood of an energy level $E$, in terms of periodic trajectories of the classical Hamiltonian system on the corresponding level set of the classical Hamiltonian. Here it is important to know whether the energy level $E$ is a regular or a critical value of the classical Hamiltonian. Therefore, we will consider these two cases separately.

\subsection{The case of a regular energy level}
In this section, we present a geometric version of Gutzwiller's formula for differential operators on manifolds and a regular energy level, following \cite{PU95}. Let $\mathcal H^\hbar$ be a differential operator on a compact manifold $M$ depending on the semiclassical parameter $\hbar>0$ of the form
\begin{equation}\label{e:Ahbar}
\mathcal H^\hbar =\sum_{l=0}^m \hbar^l A_l,
\end{equation}
where $A_l$ is a differential operator of order $l$ on $M$, $l=0,\ldots,m$. The semiclassical principal symbol of the operator $\mathcal H^\hbar$ is a smooth function on $T^*M$ given by the formula
\begin{equation}\label{e:HAhbar}
H(x, \xi) =\sum_{l=0}^m \sigma_{A_l}(x,\xi), \quad (x,\xi)\in T^*M,
\end{equation}
where $\sigma_{A_l}$ is the principal symbol of the operator $A_l$. Suppose that there is a constant $c>0$ such that $H(x,\xi)\geq c > 0$ for any $(x,\xi)\in T^*M$. Then suppose that the operator $A_m$ is elliptic in the usual sense. Finally, suppose that the operator $\mathcal H^\hbar$ is formally self-adjoint in the Hilbert space $L^2(M)$ defined by some smooth positive density on $M$.

It is easy to see that an $\hbar$-differential operator $\mathcal H^\hbar$ in the Euclidean space $\mathbb R^d$ of the form
\[
\mathcal H^\hbar =\sum_{j=0}^k \hbar^j a_j(x,\hbar D_x),
\]
where $a_j(x, D_x)$ is a differential operator of order $D$:
\[
a_j(x, D_x)=\sum_{|\alpha|\leq D}a_{j\alpha}(x)D_x^\alpha,\quad j=0,1,\ldots,k,
\]
can be written in the form \eqref{e:Ahbar} with $m=k+D$ and
\[
A_l=\sum_{j+|\alpha|=l}a_{j\alpha}(x)D^\alpha_x,\quad l=0,1,\ldots,m.
\]
Moreover, the standard definition of the principal symbol of a $\hbar$-differential operator is consistent with the formula \eqref{e:HAhbar}:
\[
H(x, \xi) =\sum_{|\alpha|\leq D} a_{0\alpha}(x)\xi^\alpha=a_0(x,\xi).
\]

As an example of an operator of the form \eqref{e:Ahbar}, we can consider the Schr\"odinger operator
\[
\mathcal H^\hbar=-\hbar^2\Delta+V,
\]
where $\Delta$ is the Laplace-Beltrami operator associated with a Riemannian metric on $M$ and $V\in C^\infty(M)$ is a strictly positive potential. 

Under these conditions, the operator $\mathcal H^\hbar$ has discrete spectrum consisting of eigenvalues $\{\lambda_j(\hbar), j=0,1,2,\ldots \}$ of finite multiplicity. Given a function $\varphi\in \mathcal S(\RR)$ and an energy level $E$, we define the smoothed spectral density $Y_{\hbar}(\varphi)$ by the formula
\begin{equation}\label{e:defYh}
Y_{\hbar}(\varphi) =\operatorname{tr} \varphi\left(\frac{\mathcal H^\hbar-E}{\hbar}\right)=\sum_{j=0}^\infty \varphi\left(\frac{\lambda_j(h)-E}{\hbar}\right).
\end{equation}

Denote by $\phi^t$ the Hamiltonian flow with the Hamiltonian $H$ on the phase space $X=T^*M$ with the canonical symplectic structure. The Gutzwiller trace formula under certain conditions on the flow $\phi^t$ gives an expression for the function $Y_{\hbar}(\varphi)$ as an asymptotic (in the semiclassical limit) series, the terms of which are expressed in terms of the geometric characteristics of the restriction of $\phi^ t$ to the level set $X_E:=H^{-1}(E)\subset T^*M$ of $H$. 

Suppose that $E$ is a regular value of the principal symbol $H$, i.e. $dH(x, \xi)\neq 0$ for any $(x,\xi) \in X_E$. Then $X_E$ is a smooth submanifold of  $T^*M$ of dimension $2d-1$.

We say that the flow $\phi^t$ is clean on $X_E$ if the set
\[
\mathcal P=\{(T,(x, \xi))\in \mathbb R\times X_E : \phi^T(x, \xi)=(x, \xi)\}
\]
is a submanifold of $\mathbb R\times X_E$ and for any $(T,(x, \xi))\in \mathcal P$ we have
\[
T_{(T,(x, \xi))}\mathcal P=\{(\tau,v)\in T_{(T,(x, \xi))}(\mathbb R\times X_E): d \phi_{(T,(x, \xi))}(\tau,v)=v\},
\]
where $\phi : \mathbb R \times X_E \to X_E, (t,(x, \xi))\mapsto \phi^t(x, \xi)$.

The action of $S_\gamma$ of a closed curve $\gamma$ in $T^*M$ is given by
\[
S_\gamma=\int_\gamma \eta,
\]
where $\eta$ is the canonical 1-form on $T^*M$:
\[
\eta=\sum_{j=1}^d\xi_jdx^j.
\]

It follows from the cleanness condition of the flow that the action is locally constant on $\mathcal P$. Let us denote its value on a connected component $\mathcal P_\nu$ of $\mathcal P$ by $\alpha_\nu$. Moreover, a canonical smooth density $d\mu_\nu$ is defined on each connected component $\mathcal P_\nu$. 

The subprincipal symbol of the operator $\mathcal H^\hbar$ is a smooth function on $T^*M$ defined by the formula
\[
H_{sub}(x, \xi) =\sum_{l=0}^m \sigma_{A_l,sub}(x,\xi),
\]
where $\sigma_{A_l,sub}$ denotes the subprincipal symbol of the operator $A_l$. We define the function $\beta$ on $\mathcal P$ by the formula
\[
\beta(T,(x, \xi))=\int_0^T H_{sub}(\phi^t(x, \xi))dt.
\]

Denote by $\hat{\varphi}$ the Fourier transform of the function $\varphi\in \mathcal S(\RR)$:
\[
\hat{\varphi}(k)=\int_{\mathbb R} \varphi(\lambda)\exp(- ik\lambda)\,d\lambda,\quad k\in \RR.
\]

The Gutzwiller trace formula is given by the following theorem. 

\begin{thm}[\cite{PU95}, Theorem 5.3]\label{t:Greg}
Assume that $E$ is a regular value of the principal symbol $H$ and the flow $\phi$ is clean on $X_E$. Then for any function $\varphi\in \mathcal S(\RR)$ whose Fourier transform is compactly supported, the smoothed spectral density $Y_\hbar(\varphi)$ admits an asymptotic expansion
\begin{multline}
\label{traceformula}
Y_\hbar(\varphi)=\sum_{j=0}^\infty \varphi\left(\frac{\lambda_j(\hbar)-E}{\hbar}\right)\\ \sim \sum_\nu e^{i\alpha_\nu \hbar^{-1}}\hbar^{-d_\nu}e^{\pi im_\nu/4} \sum_{j=0}^\infty c_{\nu,j}(\varphi)\hbar^{j},\quad \hbar\to 0,
\end{multline}
where:

(1) The sum $\sum_\nu$ is taken over all connected components $\mathcal P_\nu$ of $\mathcal P$ containing at least one point $(T,(x, \xi))$ with $T$ in the support of $\hat \varphi$. This sum is finite.

(2) $d_\nu=(\dim\mathcal P_\nu)/2$.

(3) $m_\nu$ is an integer, the common Maslov index of trajectories in $\mathcal P_\nu$.

(4) The leading coefficient of the $\nu$th term is equal to
\[
c_{\nu,0}(\varphi)=(2\pi)^{-(d_\nu+1)/2}\int_{\mathcal P_\nu}e^{i\beta}\hat{\varphi}(T)d\mu_\nu.
\]
\end{thm}

For the connected component $\{0\}\times X_E$, $d_\nu=2d-1$ and
\[
c_{\nu,0}(\varphi)=(2\pi)^{-d}\hat{\varphi}(0){\rm Vol}(X_E),
\]
where ${\rm Vol}(X_E)$ is the Liouville volume of $X_E$.

For a connected component of the form $\gamma \times \{T\}$, where $\gamma\subset X_E$ is a nondegenerate periodic trajectory of the flow, $d_\nu = 1$ and
\begin{equation}\label{e:c0-nondeg}
c_{\nu,0}(\varphi)=\frac{T^\#}{2\pi |\det(I-P_\gamma)|^{1/2}}\hat\varphi(T)
\end{equation}
where $P_\gamma$ and $T^\#_\gamma$ are the linearized Poincar\'e map and the primitive period of the trajectory $\gamma$, respectively.

The proof of Theorem \ref{t:Greg} uses the method described in Section \ref{s:inv} below. It is based on a reduction of the semiclassical spectral problem in question to some asymptotic spectral problem for joint eigenvalues of a pair of commuting pseudodifferential operators in the high-energy limit. Then the methods of the theory of Fourier integral operators and microlocal analysis developed for the proof of the Duistermaat-Guillemin trace formula are used. The cleanness condition of the flow is necessary in order to apply the theorem on the composition of Fourier integral operators, which, in its turn, essentially relies on the calculation of the asymptotics of oscillatory integrals with nondegenerate phase functions by the stationary phase method.

\subsection{The case of a critical energy level}\label{s:Gcrit}
The case where $E$ is the critical energy level was first studied in \cite{BPU95}. In this paper, the authors considered the operator $\mathcal H^\hbar$ depending on the semiclassical parameter $\hbar>0$ given by the formula \eqref{e:Ahbar} under the conditions given in the previous section. They assumed that the set
\[
\Theta = \{(x,\xi) \in T^*M : dH(x, \xi) = 0\}
\]
of critical points of the principal symbol $H$ of the operator $\mathcal H^\hbar$ is a smooth compact manifold and $H$ has a non-degenerate normal Hessian on $\Theta$, that is, for all ${(x,\xi)}\in \Theta$ the bilinear form $Q(H)_{(x,\xi)}$ on the normal space $N_{(x,\xi)}\Theta : =T_{(x,\xi)}(T^*M)/T_{(x,\xi)}\Theta $ to $\Theta$ at $(x,\xi)$ defined by the second differential $d^2_{(x,\xi)}H$ of $H$ at $(x,\xi)$ is non-degenerate. Moreover, the multiplicities of the eigenvalues of the normal Hessian $Q(H)_{(x,\xi)}$ are locally constant on $\Theta$. Without loss of generality, we may assume that $\Theta$ is connected and is contained in $X_E$ for some $E$.

Recall that $\phi^t: T^*M\to T^*M$ denotes the Hamiltonian flow with Hamiltonian $H$. It is easy to see that each point $(x,\xi)\in \Theta$ is a fixed point of the flow $\phi^t$. The differentials of the maps $\phi^t$ at ${(x,\xi)}\in \Theta$ define the linearized flow $d\phi_{t,(x,\xi)} : N_{(x,\xi) }\Theta \to N_{(x,\xi)}\Theta$. A number $T$ is a period of the flow $d\phi_{t,(x,\xi)}$ if there exists a vector $u\in T_{(x,\xi)}(T^*M)\setminus T_{ (x,\xi)}\Theta$ such that $d\phi_{t,(x,\xi)}u=u$.

Denote $q=\operatorname{codim} \Theta$, and $\nu$ is the number of negative eigenvalues of the Hessian $Q(H)$ on $\Theta$. As will be shown below, the case of the Bochner-Schr\"odinger operator corresponds to the case $q=d$ and $\nu=0$.

\begin{thm}[\cite{BPU95}, Theorem 1.1]\label{t:GBPU}
Let $\varphi\in \mathcal S(\RR)$ be any function such that $0$ is the only period of the linearized flow in the support of its Fourier transform.

(1) If $\nu\geq 1$, $q-\nu\geq 1$ and both of these numbers are odd, then the asymptotic expansion holds true
\begin{multline}\label{e:BPU1}
Y_\hbar(\varphi)=\sum_{j=0}^\infty \varphi\left(\frac{\lambda_j(\hbar)-E}{\hbar}\right)\\ \sim \hbar^{ -(d-1)}\left[\sum_{j=0}^\infty c_{j,0}\hbar^j+\sum_{j=\frac{q}{2}-1}^\infty c_ {j,1}\hbar^j[\log(1/h)]\right], \quad \hbar\to 0.
\end{multline}
(2) If $\nu\geq 1$, $q-\nu\geq 1$ and one of them is even or if the form $Q(H)$ is positive definite ($\nu=0$), then the asymptotic expansion holds true
\begin{equation}\label{e:BPU2}
Y_\hbar(\varphi)=\sum_{j=0}^\infty \varphi\left(\frac{\lambda_j(\hbar)-E}{\hbar}\right)\sim \hbar^{-( d-1)}\sum_{j=0}^\infty c_{j}\hbar^{j/2}, \quad \hbar\to 0.
\end{equation}
\end{thm}

In \cite{BPU95}, formulas for the leading coefficients of the expansions \eqref{e:BPU1} and \eqref{e:BPU2} have been also obtained. In particular, the authors have shown (see \cite[Section 3.4]{BPU95}) that if the normal Hessian is positive definite, then $\Theta$ contributes to the coefficient $c_j$ of \eqref{e:BPU2}, starting from the degree $\hbar^{-d+q/2}$ (i.e. with $j=q-2$).

The proof of Theorem \ref{t:GBPU} uses the methods of the proof of  Theorem \ref{t:Greg} with the only difference that in this case the cleanness condition of the flow is not satisfied, which leads to the necessity to study the asymptotics of oscillatory integrals with degenerate phase functions.

The proof of Gutzwiller's formula for a critical energy level, which is valid for an arbitrary function $\varphi\in \mathcal S(\RR)$, whose Fourier transform is compactly supported, and, in particular, the computation of the contributions of nonzero periods of the linearized flow, were given in \cite{ K-D97} for the Schr\"odinger operator
\[
\mathcal H^\hbar=-\frac{\hbar^2}{2}\Delta+V(x)
\]
in the Euclidean space $\mathbb R^d$ $(d\geq 1)$ provided that $$V\in C^\infty(\mathbb R^d)$$ and $\lim_{|x|\to +\infty} V(x)=+\infty$. It is well known that under such conditions the spectrum of $\mathcal H^\hbar$ is discrete. In this case, the set $\Theta$ of critical points of the principal symbol $H(x,\xi)=\frac{|\xi|^2}{2}+V(x)$ has the form
\[
\Theta = \{(x,0) \in \mathbb R^{2d} : dV(x) = 0\}.
\]
As above, suppose that $\Theta$ is a smooth compact manifold, for any $(x,0)\in \Theta$ the normal Hessian $Q(H)_{(x,0)}$ on $N_{(x,0)}\Theta$ is non-degenerate, the multiplicities of the eigenvalues of the normal Hessian $Q(H)_{(x,0)}$ are locally constant for $(x,0)\in \Theta$, $\Theta$ is connected and is contained in some $X_E$. Note that in this case $\Theta$ is an isotropic submanifold of $\mathbb R^{2d}=T^*\mathbb R^{d}$ endowed with the canonical symplectic form.

Suppose that the Hamiltonian flow $\phi^t$ in $\mathbb R^{2d}$ with Hamiltonian $H$ is clean on $X_E\setminus \Theta$. Denote by
\[
(\alpha_1(x)^2,\ldots, \alpha_r(x)^2, -\alpha_{r+1}(x)^2, \ldots, -\alpha_{r+\nu}(x)^2 ,0,\ldots,0), \quad \alpha_i(x)>0
\]
the eigenvalues of the second differential $d^2V(x)$ for $(x,0)\in \Theta$. Then $\dim \Theta = d-r-\nu$. As above, we set $q:=\operatorname{codim} \Theta = d+r+\nu$.

In \cite[Theorem 1.3]{K-D97}, analogues of the asymptotic expansions \eqref{e:BPU1} and \eqref{e:BPU2} are proved for all functions $\varphi$ whose Fourier transform is compactly supported, and the coefficients of these expansions are computed.

Let us consider in more detail the case $\nu=0$, which is the most important for us. In this case, it is proved in \cite[Theorem 1.3]{K-D97} that the asymptotic expansion \eqref{e:BPU2} remains valid. Moreover, it does not contain half-integer powers of $\hbar$, that is, it has the form
\begin{equation}\label{e:HD2bis}
Y_\hbar(\varphi)=\sum_{j=0}^\infty \varphi\left(\frac{\lambda_j(\hbar)-E}{\hbar}\right)\sim \hbar^{q/ 2-d}\sum_{j=0}^\infty c_{j}\hbar^{j}, \quad \hbar\to 0.
\end{equation}

For arbitrary $\ell\in \mathbb Z_+$, $m\in \mathbb N$ and $\alpha_j>0, j=1,\ldots, m$, define a distibution
\[
\frac{1}{(t+i0)^\ell \prod_{j=1}^m\sin \alpha_j(t+i0)}\in \mathcal S^\prime(\mathbb R)
\] 
by the formula
\begin{multline*}
\left\langle \frac{1}{(t+i0)^\ell\Pi_{j=1}^m\sin \alpha_j(t+i0)}, \psi \right\rangle \\ =\lim_{\varepsilon \to 0+}\int_{\mathbb R} \frac{\psi(t)}{(t+i\varepsilon)^\ell \prod_{j=1}^m\sin \alpha_j(t+i\varepsilon)}\,dt, \quad \psi\in \mathcal S(\mathbb R). 
\end{multline*}
The formula for the leading coefficient $c_0$ of the asymptotic expansion \eqref{e:HD2bis} has the form 
\begin{multline}\label{e:HDc0}
c_0=\frac{2^{d-q}e^{-3\pi (q/4)i}}{(2\pi)^{d-q/2}} \\ \times \int_{\Theta} \frac{1}{2\pi} \left\langle \frac{1}{(t+i0)^{d-q/2}\prod_{j=1}^{q-d}\sin ((\alpha_j(x)/2)(t+i0))}, \hat\varphi(t)\right\rangle dx. 
\end{multline}

For any $m\in \mathbb N$ and $c_j>0, j=1,\ldots, m$, the following formulas hold true (see. for instance, \cite[Lemma 3.3]{K-D97}):
\begin{equation}\label{e:KD1}
\frac{1}{2\pi} \left\langle \frac{1}{\prod_{j=1}^m\sin c_j(t+i0)}, \hat \varphi(t)\right\rangle=(-2i)^m \sum_{\mathbf k\in \mathbb Z_+^m}\varphi\left(\sum_{j=1}^m(2k_j+1)c_j\right)
\end{equation}
and for $\ell>0$
\begin{multline}\label{e:HDformula}
\frac{1}{2\pi} \left\langle\frac{1}{(t+i0)^{\ell}\prod_{j=1}^m\sin c_j(t+i0)}, \hat\varphi \right\rangle\\ =\frac{2^m e^{3\pi((\ell+m)/2) i}}{\Gamma(\ell)} \sum_{\mathbf k\in \mathbb Z_+^m}\int_{\mathbb R} \left(\tau-\sum_{j=1}^m(2k_j+1)c_j\right)_+^{\ell-1}\varphi(\tau)\,d\tau,
\end{multline}
where $\left(\tau-\beta\right)_+^{\ell-1}$ is the function, equal to $(\max(0,\tau-\beta)^{\ell-1})$. 

For $\ell\in \frac 12\mathbb N$, passing to polar coordinates, we get 
\[
\int_{\mathbb R^{2\ell}} \varphi \left(|\xi|^2+\sum_{j=1}^m(2k_j+1)c_j\right) \,d\xi =\frac{\pi^\ell}{\Gamma(\ell)} \int_{\mathbb R} \left(\tau-\sum_{j=1}^m(2k_j+1)c_j\right)_+^{\ell-1}\varphi(\tau)\,d\tau, 
\]
which allows us to rewrite \eqref{e:HDformula} as
\begin{multline}\label{e:HDformula-bis}
\frac{1}{2\pi} \left\langle\frac{1}{(t+i0)^{\ell}\prod_{j=1}^m\sin c_j(t+i0)}, \hat\varphi \right\rangle\\ =\frac{2^m e^{3\pi((\ell+m)/2) i}}{\pi^\ell} \sum_{\mathbf k\in \mathbb Z_+^m}\int_{\mathbb R^{2\ell}} \varphi \left(|\xi|^2+\sum_{j=1}^m(2k_j+1)c_j\right) \,d\xi. 
\end{multline}

Using \eqref{e:HDformula} and \eqref{e:HDformula-bis}, the formula \eqref{e:HDc0} is rewritten as 
\begin{multline}\label{e:HDc0-bis}
c_0\\ =\frac{1}{(2\pi)^{d-q/2}}\frac{1}{\Gamma(d-q/2)} \int_{\mathbb R} \left[ \int_{\Theta} \sum_{\mathbf k\in \mathbb Z_+^{q-d}} (\tau-\beta_{\mathbf k}(x))_+^{d-q/2-1}dx\right] \varphi(\tau)\,d\tau\\
=\frac{1}{(2\pi)^{2d-q}} \sum_{\mathbf k\in \mathbb Z_+^{q-d}} \int_\Theta \int_{{\mathbb R^{2d-q}}} \varphi \left(\frac{1}{2}|\xi|^2+\beta_\mathbf k(x)\right) dx\, d\xi,
\end{multline}
where
\[
\beta_{\mathbf k}(x)=\sum_{j=1}^{q-d}\left(k_j+\frac 12\right)\alpha_j(x), \quad (x,0)\in\Theta, \quad \mathbf k\in \mathbb Z_+^{q-d}.
\]

Using the last formula, we can rewrite \eqref{e:HDc0} as the classical Weyl formula with a suitably chosen operator-valued symbol. The problems of constructing analogues of the Weyl formula for various classes of degenerate operators using general theorems on the asymptotic behavior of the spectrum of pseudodifferential operators with operator-valued symbols were discussed in \cite{Lev88}. For semiclassical spectral problems, these issues are closely related to adiabatic limits and the Born-Oppenheimer approximation (See, for example, \cite{Belov-D-T,Teufel,Raymond:book} and references therein).

\subsection{Exact magnetic systems}\label{s:exact}
In this section, we return to magnetic systems and the Bochner-Schr\"odinger operator \eqref{e:BochnerHN} and consider the case of an exact magnetic system. Thus, suppose that the form $F$ is exact, $F=d\mathbf A$ with some real 1-form $\mathbf A$, the Hermitian line bundle $(L,h^L)$ is trivial, and the Hermitian connection $\nabla^L$ is written as $\nabla^L=d-i \mathbf A$. In this case, the operator $H_N$ is related to the semiclassical magnetic Schr\"odinger operator $\mathcal H^\hbar$ by the formula \eqref{e:Bochner-Hh}. Rewriting the formula \eqref{e:defYh} for the smoothed spectral density of $\mathcal H^\hbar$ at the energy level $E_0\geq 0$ in terms of $H_N$, we get
\begin{equation}
\label{e:defYp}
Y_N(\varphi)=\operatorname{tr} \varphi\left(\frac{\mathcal H^\hbar-E_0}{\hbar}\right)=\operatorname{tr} \varphi\left(\frac{1}{N}\Delta^{L^N}+V-E_0N\right). 
\end{equation}
We will use this formula to define the smoothed spectral density of the operator $H_N$ associated with an arbitrary magnetic system. 

It is not difficult to compute the prinicpal and subprincipal symbols of the operator $\mathcal H^\hbar$ (see, for example, \cite[Lemma A.1]{Morin19}): 
\begin{equation}
\label{e:H-Hsub}
H(x, \xi) = |\xi-\mathbf A(x)|_{g^{-1}}^2, \quad H_{sub}(x, \xi) = V(x), \quad (x,\xi)\in T^*M. 
\end{equation}
In local coordinates, we get
\[
H(x, \xi) = \sum_{k,\ell=1}^d g^{k\ell}(x)(\xi_k - A_k(x))(\xi_\ell - A_\ell(x)), \quad (x,\xi)\in \mathbb R^{2d}.
\]
Since
\[
\frac{\partial H}{\partial \xi_\ell}(x,\xi)=2\sum_{k,\ell=1}^d g^{k\ell}(x)(\xi_k - A_k(x)),
\]
any $E_0>0$ is a regular value of the principal symbol $H$, and the Gutzwiller formula given in Theorem \ref{t:Greg} is applicable in this case and describes a complete asymptotic expansion of $Y_N(\varphi)$ as $N \to \infty$.

The value $E_0=0$ is a critical value of $H$. We will discuss the trace formulas at this energy level later in the context of general magnetic systems. In the rest of this section, we give some facts about the geometry of the corresponding set of critical points of $H$. 

It is easy to see that the set of critical points of $H$ on the zero level set $X_0=H^{-1}(0)$ (which is often called the characteristic set of $\mathcal H^\hbar$) coincides with the entire $X_0 $. The set $X_0$ is a $d$-dimensional submanifold of $T^*M$:
\begin{equation}
\label{e:Sigma-exact}
X_0=\{(x,\xi)\in T^*M : \xi=\mathbf A(x)\}.
\end{equation}
It is identified with $M$ by the map
\[
j : x \in M \mapsto (x, \mathbf A(x)) \in X_0,
\]
and its inverse is the restriction of the projection $\pi : T^*M\to M$ to $X_0$.

It can be shown (see, for example, \cite[Lemma 2.1]{Morin21a}) that the restriction of the canonical symplectic form $\omega=\sum_{j=1}^d d\xi_j\wedge dx_j$ to $X_0$ is
\[
\omega_{X_0}=\pi^*F.
\]
Thus, in contrast to the case of the Schr\"odinger operator considered in \ref{s:Gcrit}, the submanifold $X_0$ is not isotropic. Its properties are determined by the properties of the form $F$. If $F$ is non-degenerate, then the manifold $(X_0,\omega_{X_0})$ is symplectic. If $F$ has constant rank, then $(X_0, \omega_{X_0})$ is a presymplectic manifold.

In \cite[Sections 2.1 and 2.2]{Morin21a}, the second differential $d^2H$ on $X_0$ is computed. In local coordinates $(x_1,\ldots,x_d)$ we write $$\mathbf A(x)=\sum_{k=1}^d A_k(x)dx_k$$ and introduce the notation
\[
(\nabla\mathbf A \cdot Q)_k=\sum_{\ell=1}^d \frac{\partial A_k}{\partial x_\ell}(x)Q_\ell, \quad ((\nabla\mathbf A)^T \cdot Q)_k=\sum_{\ell=1}^d \frac{\partial A_\ell}{\partial x_k}(x)Q_\ell.
\]
The tangent space $T_{j(x)}X_0$ to $X_0$ at $j(x)$ is given by
\[
T_{j(x)}X_0=\{(Q, P)\in T_{j(x)}(T^*M)\cong \mathbb R^{2d} : P=\nabla\mathbf A \cdot Q\},
\]
The skew-orthogonal complement of $T_{j(x)}X_0^\bot$ to $T_{j(x)}X_0$ has the form
\[
T_{j(x)}X_0^\bot=\{(Q, P)\in T_{j(x)}(T^*M)\cong \mathbb R^{2d} : P=(\nabla\mathbf A)^T \cdot Q\}.
\]
In particular, it is easy to see that
\[
T_{j(x)}X_0 \cap T_{j(x)}X_0^\bot= \operatorname{Ker}(\pi^*F).
\]

The second differential $d^2_{j(x)}H$ at $j(x)=(x, \mathbf A(x)) \in X_0$ is a quadratic form on $T_{j(x)}(T ^*M)\cong \mathbb R^{2d}$ given by the formula
\[
d^2_{j(x)}H(Q,P)= 2 \sum_{k,\ell=1}^d g^{k\ell}(x)(P_k - (\nabla\mathbf A \cdot Q )_k) (P_\ell - (\nabla\mathbf A \cdot Q)_\ell).
\]
Thus, the form $d^2_{j(x)}H$ defines a bilinear form $Q(H)_{j(x)}$ on $T_{j(x)}(T^*M)/T_{ j(x)}X_0$ (the normal Hessian), which is positive definite.

Let $J_x : T_xM\to T_xM$ be a skew-symmetric operator such that
\begin{equation}\label{e:defJ}
F_x(u,v)=g(J_xu,v), \quad u,v\in T_xM.
\end{equation}
Assume that the rank of the form $F_x$ is $2n$ and denote by $\pm ia_k(x)$, $k=1,\ldots,n$ $(a_k(x)>0)$ the non-zero eigenvalues of the operator $ J_x$. 

It is proved in \cite[Sections 2.1 and 2.2]{Morin21a} that there exists a linearly independent system of vectors $(\{f_k\}_{k=1}^n, \{f^\prime_{k^\prime}\}_{k^\prime=1}^n, \{g_\ell\}_{\ell=1}^{d-2n})$ in $T_{j(x)}(T^*M)\setminus T_{j(x)}X_0$ such that 
\[
\omega(f_k,f_{k^\prime})=\omega(f^\prime_k,f^\prime_{k^\prime})=0, \quad \omega(f_k,f^\prime_{k^\prime})=\delta_{k{k^\prime}}, \quad k,{k^\prime}=1,\ldots,n,
\]
\[
\omega(f_k,g_\ell)=\omega(f^\prime_k,g_\ell)= 0,\quad k=1,\ldots,n,\quad \ell=1,\ldots,d-2n, 
\]
\[
\omega(g_\ell,g_{\ell^\prime})=0,\quad \ell,{\ell^\prime}=1,\ldots,d-2n.
\]
and
\begin{multline*}
d^2_{j(x)}H(f_k,f_{k^\prime})=d^2_{j(x)}H(f^\prime_k,f^\prime_{k^\prime})=2a_k(x)\delta_{kk^\prime}, \quad d^2_{j(x)}H(f_k,f^\prime_{k^\prime})=0, \\ k,{k^\prime}=1,\ldots,n,
\end{multline*}
\[
d^2_{j(x)}H(f_k,g_\ell)=d^2_{j(x)}H(f^\prime_k,g_\ell)= 0,\quad k=1,\ldots,n,\quad \ell=1,\ldots,d-2n, 
\]
\[
d^2_{j(x)}H(g_\ell,g_{\ell^\prime})=0,\quad \ell,{\ell^\prime}=1,\ldots,d-2n.
\]
This gives a complete description of the normal Hessian in this case

\section{Trace formula at zero energy level}\label{s:trace0}

In this section, we discuss the trace formulas for the Bochner-Schr\"dinger operator \eqref{e:BochnerHN} associated with an arbitrary magnetic system described in Introduction. We will use the notation introduced in Introduction.

\subsection{Smoothed spectral density}
We use the formula \eqref{e:defYp} to define the smoothed spectral density for the Bochner-Schr\"odinger operator $H_N$ associated with an arbitrary magnetic system:
\begin{equation}
\label{e:defYp2}
Y_N(\varphi)=\operatorname{tr} \varphi\left(\frac{1}{N}\Delta^{L^N}+V-E_0N\right).
\end{equation}
If we denote by $\nu_{N,j}, j=0,1,2, \ldots,$ the eigenvalues of $H_N$  taking into account the multiplicities, then this formula takes the form
\[
Y_N(\varphi) =\sum_{j=0}^\infty \varphi\left(\frac{1}{N}\nu_{N,j}-E_0N\right).
\]

As far as we know, the question of proving the corresponding trace formula in the case of $E_0>0$ is still open (see, nevertheless, Remark \ref{r:E0}). In \cite{Gu-Uribe89} Guillemin and Uribe considered another version of the smoothed spectral density and proved the trace formula for it. A review of the main concepts and results related to the Guillemin-Uribe trace formula and some concrete examples of its computation is given in \cite{KT19} (see also \cite{KT22}). In this paper, we will pay more attention to the case of zero energy $E_0=0$ (some information for the case $E_0>0$ is given in Section \ref{s:GU}).

\subsection{Trace formula} 
For the energy level $E_0=0$, the smoothed spectral density $Y_N(\varphi)$ of the operator $H_N$ given by \eqref{e:defYp2} takes the form
\begin{equation}
\label{e:defYp2-cr}
Y_N(\varphi)=\operatorname{tr} \varphi\left(\frac{1}{N}\Delta^{L^N}+V\right)=\sum_{j=0}^\infty \varphi \left(\frac{1}{N}\nu_{N,j}\right).
\end{equation}
In \cite{Bochner-trace} the author proved the corresponding trace formula.

\begin{thm}\label{t:trace} 
There is a sequence of distributions $f_r \in \mathcal D^\prime(\mathbb R), r\geq 0$ such that for any $\varphi\in C^\infty_c(\mathbb R)$ the sequence $Y_N(\varphi)$ given by \eqref{e:defYp2-cr} admits an asymptotic expansion
\begin{equation} \label{e:main-tr}
Y_N(\varphi) \sim N^{\frac d2} \sum_{r=0}^{\infty}f_{r}(\varphi) N^{-\frac{r}{2}}, \quad N\to \infty.
\end{equation}
\end{thm}

Explicit formulas for the coefficients $f_r$ of this expansion use special differential operators $\mathcal H^{(x_0)}$ (model operators) associated with an arbitrary point $x_0\in M$. They are obtained from the operators $H_N$ by freezing the coefficients at $x_0$.

Let $x_0\in M$. We define a connection in the trivial Hermitian line bundle over $T_{x_0}M$ by the formula
\begin{equation}\label{e:nablaL0}
\nabla^{(x_0)}_{v}=\nabla_{v}+\frac{1}{2}R^L_{x_0}(w,v), \quad v\in T_w(T_{x_0} M).
\end{equation}
(recall that $R^L$ denotes the curvature of $\nabla^L $). The curvature of this connection is constant and equal to the form $R^L_{x_0}$ considered as a constant 2-form on $T_{x_0}M$. Denote by $\Delta^{(x_0)}$ the corresponding Bochner Laplacian.

The model operator $\mathcal H^{(x_0)}$ is a second-order differential operator in $C^\infty(T_{x_0}M)$ defined by the formula
\begin{equation}\label{e:DeltaL0p}
\mathcal H^{(x_0)}=\Delta^{(x_0)}+V(x_0).
\end{equation}
For any function $\varphi\in C^\infty_c(\mathbb R)$ the operator $\varphi(\mathcal H^{(x_0)})$ is an integral operator with smooth kernel $K_{\varphi(\mathcal H^ {(x_0)})}\in C^\infty(T_{x_0}M\times T_{x_0}M)$ with respect to the Euclidean volume form on $T_{x_0}M$ defined by the Riemannian metric $g_{x_0}$ .

The leading coefficient $f_{0}$ in the asymptotic expansion \eqref{e:main-tr} has the form
\begin{equation}\label{e:f0}
f_{0}=\int_Mf_{0}(x_0)dv_M(x_0),
\end{equation}
where $dv_M$ denotes the Riemannian volume form and
\begin{equation}\label{e:f0x0}
f_{0}(x_0)=K_{\varphi(\mathcal H^{(x_0)})}(0,0).
\end{equation}

The Schwartz kernel $K_{\varphi(\mathcal H^{(x_0)})}$ is easy to compute, which gives more explicit formulas for $f_{0}(x_0)$.
Recall that the skew-symmetric operator $J : TM\to TM$ is defined by the formula \eqref{e:defJ} and the nonzero eigenvalues of $J_x$ are denoted by $\pm ia_k(x), k=1,\ldots,n$ $ (a_k(x)>0)$ ($2n=\operatorname{rank} F_x$).
Put
\begin{equation}\label{e:def-Lambda}
\Lambda_{\mathbf k}(x_0)=\sum_{j=1}^n(2k_j+1) a_j(x_0)+V(x_0).
\end{equation}
In the case when $F$ has maximal rank ($d=2n$), the spectrum of $\mathcal H^{(x_0)}$ is a countable set of eigenvalues of infinite multiplicity:
\[
\sigma(\mathcal H^{(x_0)})=\left\{\Lambda_{\mathbf k}({x_0})\,:\, \mathbf k=(k_1,\cdots,k_n)\in\ZZ_+^n\right\}.
\]
If $d>2n$, the spectrum of $\mathcal H^{(x_0)}$ is a half-line:
\[
\sigma(\mathcal H^{(x_0)})=[\Lambda_0(x_0), +\infty),
\]
where
\[
\Lambda_0(x_0):=\sum_{j=1}^n a_j(x_0)+V(x_0).
\]
If $d=2n$ then
\begin{equation}\label{e:f0-1}
f_{0}(x_0) =\frac{1}{(2\pi)^{n}} \left(\prod_{j=1}^n a_j(x_0)\right) \sum_{\mathbf k\ in\ZZ_+^n}\varphi(\Lambda_{\mathbf k}(x_0)),
\end{equation}
and if $d>2n$, then
\begin{equation}\label{e:f0-2}
f_{0}(x_0) =\frac{1}{(2\pi)^{n}} \left(\prod_{j=1}^n a_j(x_0)\right) \sum_{\mathbf k\in\ZZ_+^n}\int_{\RR^{d-2n}} \varphi(|\xi|^2+\Lambda_{\mathbf k}(x_0))d\xi.
\end{equation}

Using the formulas \eqref{e:KD1}, \eqref{e:HDformula} and \eqref{e:HDformula-bis}, we can rewrite these formulas in terms of the Fourier transform of $\varphi$. In the case $d=2n$, we get
\begin{multline}\label{e:KD1a}
f_{0}(x_0)=\frac{1}{(-4i\pi)^{n}} \left(\prod_{j=1}^n a_j(x_0)\right)\\ \times \frac{1}{2\pi}\left\langle \frac{e^{itV(x_0)}}{\prod_{j=1}^n\sin a_j(x_0)(t+i0)}, \hat \varphi(t)\right\rangle, 
\end{multline}
and in the case $d>2n$
\begin{multline} \label{e:KD1b}
f_{0}(x_0) =\frac{e^{-3\pi (d/4) i}}{4^n\pi^{2n-d/2}} \left(\prod_{j=1}^n a_j(x_0)\right)\\ \times \frac{1}{2\pi}\left\langle \frac{e^{itV(x_0)}}{(t+i0)^{d/2-n}\prod_{j=1}^n\sin a_j(x_0)(t+i0)}, \hat \varphi(t)\right\rangle.
\end{multline}
For an arbitrary $r$ with $d=2n$ the coefficient $f_r(x_0)$ has the form
\[
f_r(x_0)=\sum_{\mathbf k \in\ZZ^n_+} \sum_{\ell=1}^{m} P_{\mathbf k,\ell}(x_0) \varphi^{(\ell-1)}(\Lambda_{\mathbf k}(x_0)),
\]
where $P_{\mathbf k,\ell}$ is polynomially bounded in $\mathbf k$, and for $d>2n$, 
\[
f_r(x_0)=\sum_{\mathbf k \in\ZZ^n_+}\sum_{\ell=1}^{m} \int_{\RR^{d-2n}} P_{\mathbf k,\ell,x_0}(\xi) \varphi^{(\ell-1)}(\Lambda_{\mathbf k}(x_0)+|\xi|^2)d\xi,
\]
where $P_{\mathbf k,\ell,x_0}(\xi)$ is a polynomial of degree $3r$ bounded polynomially in $\mathbf k$.

In the case of maximal rank $d=2n$, the formula \eqref{e:KD1a} has a natural geometric interpretation in terms of the magnetic geodesic flow. Since we are talking about some neighborhood of $x_0$, without loss of generality, we can assume that the magnetic system is exact, i.e. the Hermitian line bundle $L$ is trivial and there is a magnetic potential $\mathbf A$, and use the facts given in \ref{s:exact}. In particular, instead of the magnetic geodesic flow, one can consider the Hamiltonian flow $\phi^t: T^*M\to T^*M$ with the Hamiltonian $H$ given \eqref{e:H-Hsub} (see Section below \ref{s:ham}, in particular Example \ref{ex:1}).

Each point $j(x_0)=(x_0,\mathbf A(x_0))\in X_0$ is a critical point of $H$ and therefore a fixed point of the flow $\phi^t$. Thus, the linearized flow $d\phi_{t,j(x_0)}$ on the conormal space $N_{j(x_0)}X_0 := T_{j(x_0)}(T^*M)/T_{j (x_0)}X_0$ to $X_0$ at $j(x_0)$ is defined. In the case under consideration $d=2n$, the manifolds $X_0$ and $X_0^\bot$ are symplectic. There is an isomorphism $N_{j(x_0)}X_0\cong T_{j(x_0)}X_0^\bot$, and the bilinear form on $N_{j(x_0)}X_0$ induced by the canonical symplectic form $\omega$, coincides with the restriction of $\omega$ to $T_{j(x_0)}X_0^\bot$. Moreover, the flow $d\phi_{t,j(x_0)}$ is a linear Hamiltonian flow with respect to the induced symplectic structure on $N_{j(x_0)}X_0$. Its Hamiltonian is the normal Hessian $Q(H)_{j(x_0)}$, the quadratic form on $N_{j(x_0)}X_0$ defined by the second differential $d^2_{j(x_0)} H$. As indicated in Section \ref{s:exact}, the quadratic form $Q(H)_{j(x_0)}$ is positive definite, and there is a basis $\{f_k, f^\prime_k, k=1,\ldots ,n\}$ in $N_{j(x_0)}X_0$ such that
\[
\omega(f_k,f_\ell)=\omega(f^\prime_k,f^\prime_\ell)=0, \quad \omega(f_k,f^\prime_\ell)=\delta_{k\ell}, \quad k,\ell=1,\ldots,n,
\]
and
\begin{gather*}
Q(H)_{j(x_0)}(f_k, f_\ell)=
Q(H)_{j(x_0)}(f^\prime_k, f^\prime_\ell) =2a_k(x_0)\delta_{k\ell}, \\ Q(H)_{j(x_0)}(f_k, f^\prime_\ell) =0, \quad k,\ell=1,\ldots,n.
\end{gather*}
Thus, in the corresponding coordinates $(u_1,\ldots, u_n,v_1,\ldots,v_n)\in \mathbb R^{2n}$ on $N_{j(x_0)}X_0$ the linearized flow has the form $$d \phi_{t,j(x_0)}(u,v)=(u(t),v(t))\in N_{j(x_0)}X_0 \cong\mathbb R^{2n},$$ where
\begin{multline}
\begin{aligned}
\label{e:lin-flow}
u_k(t)=& \cos(2a_k(x_0)t)u_k+\sin(2a_k(x_0)t)v_k,\\
v_k(t)= & -\sin(2a_k(x_0)t)u_k+\cos(2a_k(x_0)t)v_k,\\
\end{aligned}
\quad k=1,\ldots, n, t\in \mathbb R.
\end{multline}
One can check that the expression on the right side of \eqref{e:KD1a} is written as (cf. \eqref{e:c0-nondeg})
\[
\frac{1}{\prod_{k=1}^n\sin a_k(x_0)t}=\frac{1}{|\det(I-d\phi_{t, j(x_0)})^{1/ 2}|}.
\]
In particular, the set of singularities of the Fourier transform of $f_0(x_0)$ coincides with the period set of the flow $d\phi_{t,j(x_0)}$:
\[
T=m_k\frac{\pi }{a_k(x_0)}, \quad k=1,\ldots, n, \quad m_k\in \mathbb Z.
\]
Using the relation of the magnetic geodesic flow with the flow $\phi^t$ described below in Section \ref{s:ham} (see Example \ref{ex:1}), one can easily reformulate the above facts in terms of the magnetic geodesic flow.

\subsection{Distribution of low-lying eigenvalues}
The study of the smoothed spectral density $Y_N(\varphi)$ given by  \eqref{e:defYp2-cr} is connected with the study of the asymptotic behavior of the eigenvalues of $H_N$ on intervals of the form $(\alpha N,\beta N)$ with $\alpha,\beta\geq 0$. An asymptotic formula for the eigenvalue distribution function of $\frac 1NH_N$ was proved by J.-P. Demailly \cite{Demailly85,Demailly91} using variational methods (such as Dirichlet-Neumann bracketing) without any restrictions on the curvature of $L$.

The eigenvalue distribution function $\mathcal N_N(\lambda)$ of $\frac 1NH_N$ is defined by 
\[
\mathcal N_N(\lambda)=\#\{j\in \ZZ_+: \frac 1N\nu_{N,j}\leq \lambda \},\quad \lambda\in \RR,
\]
where $\nu_{N,j}, j\in \ZZ_+$ are the eigenvalues of $H_N$ taking into account the multiplicities. By \cite[Theorem 0.6]{Demailly85} (see also \cite[Corollary 3.3]{Demailly91}), there exists a countable set $\mathcal D\subset \RR$ such that for any $\lambda\in \RR\setminus \mathcal D$
\begin{multline}\label{e:Demailly1}
\lim_{N\to +\infty}N^{-d/2}\mathcal N_N(\lambda)\\ =\frac{2^{n-d}\pi^{-d/2}}{\Gamma(\frac d2-n+1)} \sum_{\mathbf k\in \ZZ_+^n} \int_M (\lambda-\Lambda_{\mathbf k}({x}))_+^{d/2-n}\left(\prod_{j=1}^n a_j(x)\right) dv_M(x).
\end{multline}
It is easy to see that this formula agrees with the formulas
\eqref{e:f0-1} and \eqref{e:f0-2} (see also \eqref{e:HDc0} and \eqref{e:HDc0-bis}).

In the case of maximal rank $d=2n$, the formula \eqref{e:Demailly1} can be rewritten in terms of the Liouville volume form $\mu_F=\frac{1}{n!} F^n$ as follows:
\begin{equation}\label{e:Demailly2}
\lim_{p\to +\infty}N^{-n}\mathcal N_N(\lambda)=\frac{1}{(2\pi)^n} \int_M \#\{\mathbf k\in \ZZ_+^n : \Lambda_{\mathbf k}({x}) <\lambda \}\mu_F(x). 
\end{equation}
 
The proof of this formula, based on the use of the heat equation, was given in \cite{bouche90}. We refer the reader to the papers \cite{bismut87,bouche90,Demailly91,ma-ma:book,ma-ma-zelditch15} and the bibliography given there for the study of the heat kernel associated $\frac 1NH_N$, as well as to the papers \cite{Morin19,charles21}, where the asymptotic Weyl formula was studied in the case when the curvature $R^L$ is non-degenerate. There is also an extensive literature devoted to studies of the asymptotic behavior of the low-lying eigenvalues of $H_N$. See, for example, books and review papers \cite{FH10,HK-luminy,HK14,Raymond:book}, as well as recent papers \cite{ma-savale,Morin19,Morin21a} (and the references therein).

In \cite{higherLL,charles21} an asymptotic description of the spectrum of $\frac 1NH_N$ is given in terms of the spectra of the model operators \eqref{e:DeltaL0p} in the case when the form $F$ has maximal rank. Denote by $\Sigma$ the union of the spectra of the model operators:
\begin{equation*}
\Sigma=\left\{\Lambda_\mathbf {k}(x_0)\,:\, \mathbf k\in\ZZ_+^n, x_0\in M \right\}.
\end{equation*}

\begin{theorem}[\cite{charles21}]\label{t:spectrum}
For any $K>0$ there exists $c>0$ such that for any $N\in \NN$ the spectrum of $\frac 1NH_N$ in the interval $[0,K]$ is contained in the $cN^{-1/ 2}$-neighborhood of $\Sigma$.
\end{theorem}

In \cite{higherLL} a similar statement is proved with a weaker estimate $cN^{-1/4}$ instead of $cN^{-1/2}$ for a wider class of Riemannian manifolds of bounded geometry.

The set $\Sigma$ is a closed subset of the real line $\RR$, which can be represented as a union of closed intervals 
\[
\Sigma=\bigcup_{\mathbf k\in\ZZ_+^n}[\alpha_{\mathbf k}, \beta_{\mathbf k}]
\]
where for any $\mathbf k\in\ZZ_+^n$ the interval $[\alpha_{\mathbf k}, \beta_{\mathbf k}]$ is the image of the function $\Lambda_{\mathbf k}$ on $M $: $[\alpha_{\mathbf k}, \beta_{\mathbf k}]=\{\Lambda_{\mathbf k}({x_0}) : x_0\in M\}$. In general, the bands $[\alpha_{\mathbf k}, \beta_{\mathbf k}]$ can overlap without any gaps, and then $\Sigma$ is the semiaxis $[\Lambda_0,+\infty)$ with $ \Lambda_0=\inf_{x\in M} \Lambda_0(x)$. In some cases $\Sigma$ may have gaps: $[\Lambda_0,+\infty)\setminus \Sigma\not=\emptyset$. For example, if $V(x)\equiv 0$ and functions $a_j$ can be chosen constant:
\begin{equation}\label{e:aj-constant}
a_j(x)\equiv a_j, \quad x\in M, \quad j=1,\ldots,n,
\end{equation}
then $\Sigma$ is a countable discrete set. In particular, if $J$ is an almost-complex structure ($J^2=-I$; almost-K\"ahler case) and $V(x)\equiv 0$, then $a_j=1, j=1,\ldots ,n$, and
\[
\Sigma=\left\{2k+n\,:\, k\in\ZZ_+\right\}.
\]
The set $\Sigma$ can also have gaps if the functions $a_j$ are not constant, but change little enough. In these cases, Theorem \ref{t:spectrum} also implies the existence of gaps in the spectrum of $\frac 1NH_N$. In particular, if $V(x)\equiv 0$ and the condition \eqref{e:aj-constant} is true, then the spectrum of $\frac 1NH_N$ is contained in the union of neighborhoods of points $a_j$ of size $O(N^{- 1/2})$. In the almost-K\"ahler case, Theorem \ref{t:spectrum} was proved in \cite{FT}. 

The spectral data of the operator $H_N$ can be used to construct the Berezin-Toeplitz quantizations of the symplectic manifold $(M,F)$. The space of such a quantization is the spectral subspace of the operator $\frac 1NH_N$ corresponding to the eigenvalues located near some isolated closed component of  $\Sigma$, and the quantization operators are the Toeplitz operators associated with this subspace. This idea was first proposed by Guillemin and Uribe in \cite{Gu-Uribe}. They considered the Bochner-Schr\"odinger operator with the potential $V(x)=-\tau(x)$, where $\tau(x):=\operatorname{Tr} |J_x|, x\in M$, called the renormalized Bochner Laplacian, $\Delta_N:=\Delta^{L^N}-N\tau$.
  
An important special case (and motivation for such a definition) is the case of a K\"ahler manifold $M$. If we take $L$ to be a holomorphic line bundle on $M$ endowed with a holomorphic connection (Chern connection), then the renormalized Bochner Laplacian coincides with the twice Kodaira Laplacian, $\Delta_N=2(\bar\partial^{L^N}) ^*\bar\partial^{L^N}$. The quantum space in this case consists of holomorphic sections of $L^N$. The corresponding Berezin-Toeplitz quantization is called K\"ahler quantization.

In a general case, it was proved in \cite{Gu-Uribe} that there exist constants $c>0$ and $b_0>0$ such that for any $N\in \NN$ the spectrum of the renormalized Bochner Laplacian $\Delta_N$ is contained in $(-c, c)\cup [2b_0N-c, \infty)$. A simpler proof and precise expression for the constant $b_0$ was given in \cite[Corollary 1.2]{ma-ma02}. Note that this result agrees with Theorem \ref{t:spectrum}, since in this case $\Lambda_0(x)\equiv 0$ and $\Sigma$ has a gap near zero: $\Sigma\subset\{0\} \cup [2b_0, \infty)$. The quantum space is generated by the eigenfunctions of  $H_N$ with eigenvalues from the interval $(-c, c)$.
The corresponding quantization was constructed in \cite{B-Uribe} in the almost-K\"ahler case and in \cite{ioos-lu-ma-ma,Kor18} for an arbitrary Riemannian metric (see also \cite{BMS94} about K\"ahler quantization and \cite{B-Uribe,ma-ma08a,ma-ma:book} for the quantization associated with the spin$^c$ Dirac operator). In \cite{charles20a,Kor20a}, Berezin-Toeplitz quantization are constructed for more general spectral subspaces corresponding to arbitrary isolated closed components of $\Sigma$.

\subsection{Proof by methods of local index theory}
In this section, we briefly describe the main steps in the proof of Theorem \ref{t:trace} following \cite{Bochner-trace}.
The proof combines the methods of functional analysis (first of all, functional calculus based on the Helffer-Sj\"{o}strand formula \cite{HS-LNP345} and norm estimates in suitable Sobolev spaces) with the methods of local index theory developed in \cite{dai-liu-ma,ma-ma:book,ma-ma08} for the study of the asymptotic behavior of the (generalized) Bergman kernels and originating in the paper by Bismut-Lebeau \cite{BL}. Note that, in contrast to \cite{dai-liu-ma,ma-ma:book,ma-ma08}, we do not require that the curvature of $L$ be non-degenerate. A similar strategy was applied in a close situation by N. Savale in \cite{Savale17}. 

First of all, we localize the problem in a neighborhood of an arbitrary point $x_0\in M$ using the constructions from \cite[Sections 1.1 and 1.2]{ma-ma08}. Denote by $B^{M}(x_0,r)$ and $B^{T_{x_0}M}(0,r)$ the open balls in $M$ and $T_{x_0}M$ centered at $x_0 $ with radius $r$, respectively. Let $r_M$ denote the injectivity radius of the Riemannian manifold $(M,g)$. We will identify the balls $B^{T_{x_0}M}(0,r_M)$ and $B^{M}(x_0,r_M)$ using the exponential map $\exp^M_{x_0}: T_{x_0} M \to M$. Let us choose a trivialization of the bundle $L$ over $B^M(x_0,r_M)$ by identifying its fiber $L_Z$ at the point $Z\in B^{T_{x_0}M}(0,r_M)\cong B^M( x_0,r_M)$ with the fiber $L_{x_0}$ at $x_0$ by means of the parallel transport given by the connection $\nabla^L$ along the curve $\gamma_Z : [0,1]\ni u \to \exp ^M_{x_0}(uZ)$. Consider the trivial Hermitian line bundle $L_0$ on $T_{x_0}M$ with fiber $L_{x_0}$. The above identifications induce a Riemannian metric $g$ on $B^{T_{x_0}M}(0,r_M)$, as well as a connection $\nabla^L$ and a Hermitian metric $h^L$ on the restriction of $L_0$ to $B^{T_{x_0}X}(0,r_M)$.

Now we fix some $\varepsilon \in (0,r_M)$ and extend all the geometric objects introduced above from $B^{T_{x_0}M}(0,\varepsilon)$ to $T_{x_0}M$ as follows . Let $\rho : \mathbb R\to [0,1]$ be a smooth finite even function supported in $(-r_M,r_M)$ such that $\rho(v)=1$ if $|v|<\varepsilon$. Consider the map $\varphi : T_{x_0}M\to T_{x_0}M$ defined by $\varphi (Z)=\rho(|Z|)Z$. We define the Riemannian metric $g^{(x_0)}$ on $T_{x_0}M$ by $g^{(x_0)}_Z=g_{\varphi(Z)}, Z\in T_{x_0}M$ , and the Hermitian connection $\nabla^{L_0}$ on $(L_0,h^{L_0})$ by 
\[
\nabla^{L_0}_u=\nabla^L_{d\varphi(Z)(u)}. \quad Z\in T_{x_0}M, \quad u\in T_Z(T_{x_0}M),
\]
where we use the canonical isomorphism $T_{x_0}M\cong T_Z(T_{x_0}M)$. Finally, we set $V^{(x_0)}=\varphi^*V$.
Let $\Delta^{L_0^N}$ denote the associated Bochner Laplacian on $C^\infty(T_{x_0}M, L_0^N)$.

We introduce the operator $H^{(x_0)}_N$ on $C^\infty(T_{x_0}M,L_0^N)$ by the formula
\[
H^{(x_0)}_N=\Delta^{L_0^N}+ NV^{(x_0)}.
\]
It is easy to see that for any function $u \in {C}^\infty_c(T_{x_0}M)$ supported in $B^{T_{x_0}X}(0,\varepsilon)$,
\begin{equation}\label{e:Hp=HX0p}
H_Nu(Z)=H^{(x_0)}_Nu(Z), \quad Z\in T_{x_0}M.
\end{equation}

Let $K_{\varphi(\frac 1N H^{(x_0)}_N)}\in {C}^{\infty}(T_{x_0}M\times T_{x_0}M)$ be the Schwartz kernel of $\varphi(\frac 1N H^{(x_0)}_N)$ with respect to the Riemannian volume form $dv^{(x_0)}$ on $(T_{x_0}M, g^{(x_0)}) $. Writing the Schwartz kernel $K_{\varphi(\frac 1NH_N)}$ of $\varphi(\frac 1NH_N)$ in local coordinates, we obtain a family of smooth functions
\[
K_{\varphi(\frac 1NH_N),x_0}\in {C}^{\infty}(B^{T_{x_0}M}(0,r_M)\times B^{T_{x_0}M}(0 ,r_M))
\]
parametrized by $x_0\in M$:
\begin{multline}\label{e:defKvarphi}
K_{\varphi(\frac 1N H_N),x_0}(Z,Z^\prime)=K_{\varphi(\frac 1N\varphi(H_N))}(\exp^M_{x_0}(Z),\ exp^M_{x_0}(Z^\prime)), \\ Z,Z^\prime\in B^{T_{x_0}M}(0,r_M).
\end{multline}
Using the equality \eqref{e:Hp=HX0p} and the finite propagation speed property, one can show that for any $\varepsilon_1\in (0,\varepsilon)$ and $k\in \mathbb N$ there exists $C>0 $ such that
\begin{equation} \label{e:K-K}
|K_{\frac 1N\varphi(H_N),x_0}(Z,Z^\prime)-K_{\varphi(\frac 1NH^{(x_0)}_N)}(Z,Z^\prime)|\leq CN^{-k}
\end{equation}
for any $N\in\NN$, $x_0\in M$ and $Z,Z^\prime\in B^{T_{x_0}M}(0,\varepsilon_1)$. A similar estimate is also valid for covariant derivatives of any order with respect to $x_0$.
This fact allows us to reduce our considerations to the case of the $C^\infty$-bounded family $\frac 1NH^{(x_0)}_N$ of second-order differential operators acting on $C^\infty(T_{x_0}M, L_0^ N)\cong C^\infty(T_{x_0}M)$ (parametrized by $x_0\in M$).

We will now use a scaling introduced in \cite[section 1.2]{ma-ma08}. Denote $t=\frac{1}{\sqrt{N}}$ and for $s\in C^\infty(T_{x_0}M)$ set
\[
S_ts(Z)=s(Z/t), \quad Z\in T_{x_0}M.
\]
Let $dv_{M,x_0}$ denote the Riemannian volume form of the Euclidean space $(T_{x_0}M, g_{x_0})$. We define a smooth function $\kappa_{x_0}$ on $B^{T_{x_0}M}(0,r_M)\cong B^{M}(x_0,r_M)$ using the equation
\[
dv_M(Z)=\kappa_{x_0}(Z)dv_{M,x_0}(Z), \quad Z\in B^{T_{x_0}M}(0,r_M).
\]
Let us define the transformation of $\frac 1N H_N^{(x_0)}$ by the formula
\begin{equation}\label{scaling}
\mathcal H_t=S^{-1}_t\kappa_{x_0}^{\frac 12}\frac 1N H_N^{(x_0)}\kappa_{x_0}^{-\frac 12}S_t,
\end{equation}
By definition, the operator $\mathcal H_t$ is a self-adjoint operator in $L^2(T_{x_0}M)$, and its spectrum coincides with the spectrum of $\frac 1N H_N^{(x_0)}$.

An arbitrary orthonormal basis $\mathbf e=\{e_j, j=1,2,\ldots,d\}$ in $T_{x_0}M$ defines an isomorphism $T_{x_0}M\cong \mathbb R^{d} $ and allows us to transfer the operator $\mathcal H_t$ to $L^2(R^{d})$.
Thus, we obtain a family of self-adjoint differential operators on $C^\infty(\RR^d)$ depending smoothly on $\mathbf e$, which we will also denote by $\mathcal H_t$, omitting the index $\mathbf e$.

It can be shown that the operators $\mathcal H_t$ depend smoothly on $t$ up to $t=0$, and their limit at $t\to 0$ coincides with the operator $\mathcal H^{(x_0)}$ given by \eqref{e:DeltaL0p}.
Expanding the coefficients of $\mathcal H_t$ into a Taylor series in $t$, for any $m\in \NN$, we get
\begin{equation}\label{e:Ht-formal}
\mathcal H_t=\mathcal H^{(0)}+\sum_{j=1}^m \mathcal H^{(j)}t^j+\mathcal O(t^{m+1}), \quad \mathcal H^{(0)}=\mathcal H^{(x_0)},
\end{equation}
where there exists $m^\prime\in \NN$ such that for any $k\in\NN$ and $t\in [0,1]$ all the derivatives of the coefficients of $\mathcal O(t^{m+1 })$ up to order $k$ are bounded by $Ct^{m+1}(1+|Z|)^{m^\prime}$.

The operators $\mathcal H^{(j)}, j\geq 1,$ have the form (see \cite[Theorem 1.4]{ma-ma08}):
\begin{equation}\label{e:Hj}
\mathcal H^{(j)}=\sum_{k,\ell=1}^{d} a_{k\ell,j}\frac{\partial^2}{\partial Z_k\partial Z_\ell} +\sum_{k=1}^{d} b_{k,j}\frac{\partial}{\partial Z_k}+c_{j},
\end{equation}
where $a_{k\ell,j}$ is a homogeneous polynomial in $Z$ of degree $j$, $b_{kj}$ is a polynomial in $Z$ of degree $\leq j+1$ (of the same parity as $j-1$) and $c_{j}$ is a polynomial in $Z$ of degree $\leq j+2$ (of the same parity as $j$). In \cite[Theorem 1.4]{ma-ma08}, explicit formulas for the operators $\mathcal H^{(1)}$ and $\mathcal H^{(2)}$ are given.

Now we apply the Helffer-Sj\"{o}strand formula \cite{HS-LNP345}:
\begin{equation}\label{e:HS}
\varphi(\mathcal H_{t})=-\frac{1}{\pi}
\int_\CC \frac{\partial \tilde{\varphi}}{\partial \bar \lambda}(\lambda)(\lambda-\mathcal H_{t})^{-1}d\mu d\nu,
\end{equation}
where $\tilde{\varphi}\in C^\infty_c(\CC)$ is an almost-analytic extension of $\varphi$ satisfying the condition
\[
\frac{\partial \tilde{\varphi}}{\partial \bar \lambda}(\lambda)=O(|\nu|^\ell),\quad \lambda=\mu+i\nu, \quad \nu\to 0,
\]
for any $\ell\in \NN$. Using this formula, estimates of the resolvents $(\lambda-\mathcal H_{t})^{-1}$ in Sobolev spaces, and the Sobolev embedding theorem in an appropriate way, one can prove that the Schwartz kernel $K_{\varphi(\mathcal H_{ t})}(Z,Z^\prime)$ of $\varphi(\mathcal H_{t})$ is a smooth function of the variables $Z,Z^\prime\in \RR^{d}$ and $t\geq 0$ (depending smoothly on $\mathbf e$). Therefore,  Taylor's formula implies an asymptotic expansion 
\begin{equation}\label{e:Kft}
K_{\varphi(\mathcal H_{t})}(Z,Z^\prime) \sim \sum_{r=0}^\infty F_{r}(Z,Z^\prime) t^r, \quad t\to 0+,
\end{equation}
with some $F_{r}=F_{r,\mathbf e}\in C^\infty(\mathbb R^{d}\times \mathbb R^{d})$, uniform in $\mathbf e$. 

By \eqref{scaling}, we have 
\[
K_{\varphi(\frac 1NH^{(x_0)}_p)}(Z,Z^\prime)=t^{-d}\kappa^{-\frac 12}(Z)K_{\varphi(\mathcal H_{t})}(Z/t,Z^\prime/t)\kappa^{-\frac 12}(Z^\prime), \quad Z,Z^\prime \in \mathbb R^{ d}.
\]
Therefore, \eqref{e:Kft} and \eqref{e:K-K} imply the existence of an asymptotic expansion of the kernel $K_{\varphi(\frac 1NH_N)}$ of $\varphi(\frac 1NH_N)$ on the diagonal, uniform in $x_0$,
\[
K_{\varphi(\frac 1NH_N)} (x_0,x_0) \sim N^{\frac d2} \sum_{r=0}^\infty f_{r}(x_0)N^{-\frac{r} {2}},\quad N\to \infty,\quad x_0\in M,
\]
where
\[
f_{r}(x_0)=F_{r,\mathbf e}(0,0).
\]
for any orthonormal frame $\mathbf e$ at $x_0$. This immediately implies the asymptotic expansion \eqref{e:main-tr} with
\[
f_r=\int_Mf_r(x_0)dv_M(x_0),
\]
Moreover, for $r=0$ we have
\[
F_0= -\frac{1}{\pi }\int_\CC \frac{\partial \tilde{\varphi}}{\partial \bar \lambda}(\lambda) (\lambda - \mathcal H^{( 0)})^{-1} d\mu d\nu=\varphi(\mathcal H^{(0)}),
\]
which proves the formula \eqref{e:f0x0}.

Note that, using the technique of weighted Sobolev spaces, we can prove asymptotic expansions for the kernel $K_{\varphi(\frac 1N H_N),x_0}(Z,Z^\prime)$ defined by \eqref{e:defKvarphi}, in some fixed neighborhood of the diagonal, i.e. for any $x_0\in M$ and $Z,Z^\prime\in B^{T_{x_0}M}(0,\varepsilon)$ with some $\varepsilon>0$. Such expansions are  generalizations of asymptotic expansions for (generalized) Bergman kernels proved in \cite[Theorem 4.18']{dai-liu-ma}, \cite[Theorem 4.2.1]{ma-ma:book} and \cite[Theorem 1]{Kor18}. They are often called full off-diagonal expansions, following the book by Ma and Marinescu \cite[Chapter 4]{ma-ma:book}. We refer the reader to \cite{Bochner-trace} for more details. 

\section{Trace formulas and methods of microlocal analysis} \label{s:trace1}

In this section, we describe another approach to proving the trace formula for an arbitrary magnetic system, using methods of microlocal analysis. It is based on the idea first proposed by Colin de Verdi\`ere \cite{CdV79} to study semiclassical spectral problems and used in \cite{PU95,BPU95} to prove Gutzwiller's formula. This idea consists in reducing the semiclassical spectral problem to some asymptotic problem for the joint eigenvalues of a pair of commuting pseudodifferential operators and then applying well-developed methods for studying high energy spectral asymptotics. This approach was subsequently extended to the problem under consideration in \cite{Gu-Uribe89,BU07}. We will use the notation introduced in Introduction.

\subsection{Reduction to the case of commuting operators}\label{s:inv}
The main idea is to interpret the semiclassical parameter $N$ as an eigenvalue of the operator $D_\theta=i^{-1}\partial/\partial\theta$ on the unit circle $S^1=\RR/2\pi \ZZ$ with coordinate $\theta$, considering $\theta$ as an additional independent variable. In local coordinates, this means that it is proposed to consider instead of the operator \eqref{e:DLp} the so-called horizontal Laplacian $\Delta_{\rm h}$ on $M\times S^1$ given by the formula
\begin{multline*}
\Delta_{\rm h}=-\frac{1}{\sqrt{|g(x)|}}\sum_{1\leq j,\ell\leq n}\left(\frac{\partial}{\partial x^j}-A_j(x)\frac{\partial}{\partial \theta}\right)\times \\ \times \left[\sqrt{|g(x)|}
g^{j\ell}(x) \left(\frac{\partial}{\partial
x^\ell}-A_\ell(x)\frac{\partial}{\partial \theta}\right)\right]+\frac{N}{i}\frac{\partial}{\partial \theta}.
\end{multline*}

In the general case, let us consider the principal $S^1$-bundle $\Pi : S\to M$ associated with $L$:
\[
S=\{p\in L^* : |p|_{h^{L^*}}=1\}.
\]
Denote by $e^{i\theta}\cdot p$ the action of $\theta\in S^1$ on $p\in S$ given by complex multiplication in the fibers of $L$. Let $\partial/\partial\theta$ denote the infinitesimal generator of the $S^1$-action on $S$. 

The connection $\nabla^L$ induces a connection on the principal bundle $\Pi : S\to M$, that is, the real-valued connection $1$-form $\alpha$ on $S$ and the $S^1$-invariant distribution $H\subset TS$ transversal to the fibers of $\Pi$ (the horizontal distribution of the connection). Thus, for any $p\in S$, the tangent space $T_pS$ can be represented as a direct sum of subspaces
\[
T_pS = V_p\oplus H_p,
\]
where $V_p$ is the tangent space to the fiber of $\Pi$ (it is generated by the vector $\partial/\partial\theta$) and $H_p$ is the horizontal space of the connection. We define the Riemannian metric $g_S$ on $S$ as follows (the Kaluza-Klein metric). The metric $g_S$ on $V_p$ coincides with the standard Riemannian metric $d\theta^2$ on $S^1=\mathbb R/2\pi\mathbb Z$. The restriction of the metric $g_S$ to $H_p$ is consistent with the Riemannian metric $g$ to $M$ under the linear isomorphism
\begin{equation}\label{e:dPi}
d\Pi_p : H_p\subset T_pS\stackrel{\cong}{\to} T_xM, \quad x=\Pi(p),
\end{equation}
defined by the differential of $\Pi$. Finally, the subspaces $V_p$ and $H_p$ are orthogonal. The projection $\Pi : (S,g_S) \to (M,g)$ is a Riemannian submersion with totally geodesic fibers.

Given the distribution $H$ and the Riemannian metric $g_S$, one naturally constructs the horizontal Laplacian $\Delta_{\rm h}$, which is a second-order differential operator on $S$. Denote by $\Omega^1$ the space of smooth differential 1-forms on $S$. For any $p\in S$, there is a decomposition of  $T^*_pS$ as a direct sum of subspaces
\begin{equation}\label{e:decomp*}
T^*_pS =V^*_p\oplus H^*_p,
\end{equation}
and the corresponding decomposition of the space of smooth differential 1-forms:
\[
\Omega^1=\Omega^1_V \oplus \Omega^1_H, \quad \Omega^1_V=C^\infty(S,V^*), \quad \Omega^1_H=C^\infty(S,H ^*).
\]
Denote by $d_{\rm h} : C^\infty(S)\to \Omega^1_H$ the composition of the de Rham differential $d: C^\infty(S)\to \Omega^1$ and the projection $\Omega ^1$ to $\Omega^1_H$. The horizontal Laplacian $\Delta_{\rm h}$ is defined by the formula
\[
\Delta_{\rm h}=d^*_{\rm h}d_{\rm h}.
\]

The operator $\Delta_{\rm h}$ is not an elliptic operator. Its principal symbol is given by the formula
\[
\sigma(\Delta_{\rm h})(p,\nu)=|\nu_H|^2, \quad p\in S, \quad \nu\in T^*_pS,
\]
where $\nu_H\in H^*_p$ is the component of $\nu\in T^*_pS$ in the decomposition \eqref{e:decomp*}. The subspace $V^*_p$ is one-dimensional and is generated by the connection form $\alpha_p$. Therefore, the characteristic set of the operator $\Delta_{\rm h}$, i.e. the set of zeros of its principal symbol $\sigma(\Delta_{\rm h})$ has the form
\begin{equation}\label{e:char1}
\mathcal Z=V^*=\{(p,r\alpha_p)\in T^*S : p\in S, r\in \mathbb R\}.
\end{equation}
It is a $d+2$-dimensional homogeneous submanifold of $2d+2$-dimensional manifold $T^*S$. 
 
The eigenvalues of $\Delta_{\rm h}$ are described as follows. For any $N\in \ZZ$ consider the space $E_N$ of smooth functions on $S$ such that
\[
f(e^{i\theta}\cdot p)=e^{iN\theta}f(p)\ \text{for any}\ p\in S\ \text{and}\ \theta\in S^1.
\]
There is an isomorphism
\begin{equation}\label{e:iso}
C^\infty(M,L^N)\cong E_N, \quad N\in \mathbb Z,
\end{equation}
which assigns to each $s\in C^\infty(M,L^N)$ the function $\hat s\in C^\infty(S)$ given by  
\begin{equation}\label{e:iso1}
\hat s(p)=\langle s(\Pi(p)), p^{\otimes N}\rangle,\quad p\in S\subset L^*.
\end{equation}
Under the isomorphism \eqref{e:iso}, the restriction of $\Delta_{\rm h}$ to the subspace $E_N$ corresponds to the magnetic Laplacian $\Delta^{L^N}$. Therefore, the spectrum of the operator $\Delta_{\rm h}$ is the union of the spectra of the operators $\Delta^{L^N}$ over all $N\in \mathbb Z$:
\[
{\rm spec}(\Delta_{\rm h})=\{\nu_{N,j} : j\in \NN, N\in \ZZ\}. 
\]

On the other hand, the eigenvalues of the first order differential operator $D_\theta=\frac{1}{i}\frac{\partial}{\partial \theta}$ are integers and the eigenspace corresponding to the eigenvalue $N\ in \ZZ$ is the space $E_{N}$. Since the $S^1$-action on $S$ is isometric, the operator $\Delta_{\rm h}$ commutes with $D_\theta$. The joint eigenvalues of the operators $\Delta_{\rm h}$ and $D_\theta$ are $\{(\nu_{N,j},N) , j\in \NN, N\in \ZZ\}$.

These facts allow us to express the smoothed spectral density of the Bochner-Schr\"odinger operator $H_N$ in terms of the joint spectral characteristics of some commuting operators, which will be discussed later in Sections \ref{s:GU} and \ref{s:BU}.

\subsection{Hamiltonian reduction and magnetic geodesic flow}\label{s:ham}
The above construction of lifting to the bundle $S$ also allows us to give a natural definition of the classical dynamical system associated with the magnetic Laplacian, i.e. magnetic geodesic flow. Namely, the magnetic geodesic flow $\Phi$ on $T^*M$ coincides with the Hamiltonian reduction of the Riemannian geodesic flow $f$ on $T^*S$ given by the Riemannian metric $g_S$. Let us briefly recall this well-known construction (see, for instance, \cite{Gu-Uribe89}, \cite[Section 6.6]{OR04} and references therein). We will use the notation introduced in the previous section.

Recall that the magnetic geodesic flow $\Phi^t : T^*M\to T^*M$ associated with a magnetic system $(g,F)$ is the Hamiltonian flow given by the Hamiltonian
\begin{equation}\label{e:1.2a}
\cH(x,\xi)=\frac{1}{2}|\xi|^2_{g^{-1}}=\frac{1}{2}\sum_{j,k=1}^ng^ {jk}\xi_j\xi_k,
\end{equation}
with respect to the twisted symplectic form on $T^*M$:
\begin{equation}\label{e:1.3}
\Omega_F=\omega+\pi^*_M F.
\end{equation}
Here $\omega$ is the canonical symplectic form on $T^*M$ and $\pi_M :T^*M\to M$ is the canonical projection.

The $S^1$-action on $S$ defines an $S^1$-action on $T^*S$. This action is Hamiltonian with the corresponding momentum map $\mu : T^*S\to T^*_{0}S^1\cong \mathbb R$ given by
\[
\mu(p,\nu)=\left\langle\nu, \frac{\partial}{\partial\theta}\right\rangle,\quad (p,\nu)\in T^*S.
\]
Consider the submanifold
\[
\mu^{-1}(1)=\left\{\nu \in T^*S : \left\langle \nu, \frac{\partial}{\partial\theta}\right\rangle=1\right\}.
\]
It is easy to see that it is $S^1$-invariant. The reduced symplectic manifold $B$ is defined as the manifold of orbits of the induced $S^1$-action on $\mu^{-1}(1)$, $B=\mu^{-1}(1)/S^1$. The reduced symplectic form on $B$ is naturally defined by the restriction of the canonical symplectic form on $T^*S$ to the submanifold $\mu^{-1}(1)$.

The manifold $B$ is diffeomorphic (but not canonically) to the cotangent bundle $T^*M$. A diffeomorphism $T^*M\cong B$ can be constructed by choosing a connection $\alpha$ on the bundle $\Pi : S\to M$. Namely, for any $p\in S$ there is a linear isomorphism
\[
d\Pi^*_p : T^*_xM\stackrel{\cong}{\to} \mathcal H^*_p,\quad \Pi(p)=x,
\] dual to \eqref{e:dPi}. The orbit of the $S^1$-action on $\mu^{-1}(1)$ corresponding to $(x,\xi)\in T^*M$ has the form 
\[
\{\alpha_p+d\Pi^*_p (x,\xi)\in \mu^{-1}(1) : p\in S, \Pi(p)=x\}.
\]

The map $\tilde\Pi : \mu^{-1}(1)\to T^*M$ given by $\tilde\Pi(p,\nu)=(x,\xi)$, where $x =\Pi(p)$ and $d\Pi^*_p (x,\xi)=\nu-\alpha_p$, defines a principal $S^1$-bundle over $T^*M$. One can check that the restriction of the canonical 1-form $\eta_S$ on $T^*S$ to $\mu^{-1}(1)$ defines a connection form on the bundle $\tilde\Pi : \mu^{- 1}(1)\to T^*M$. Moreover, the curvature of this connection coincides (up to the factor $i$) with the twisted symplectic form $\Omega_F$ on $M$ given by 
\eqref{e:1.3}. This easily implies that the reduced symplectic structure on $T^*M$ is given by the form $\Omega_F$. 

If $\alpha^\prime$ is another connection form on $\Pi : S\to M$, then the diffeomorphisms $B\cong T^*M$ given by the forms $\alpha$ and 
$\alpha^\prime$ are related as follows. It is well known that
$\alpha^\prime-\alpha=\Pi^*\sigma$ for some 1-form $\sigma$ on $M$. We define a map $T : T^*M\to T^*M$ by 
\begin{equation}\label{e:defT}
T(x,\xi)=(x, \xi -\sigma_x), \quad (x,\xi)\in T^*M.
\end{equation}
It is easy to check that the map $T$ is compatible with the diffeomorphisms $B\cong T^*M$ determined by the forms $\alpha$ and $\alpha^\prime$, and
\[
T^*\Omega_F=\Omega_{F^\prime},
\]
where $F^\prime=F+d\sigma$ is the curvature of the connection $\alpha^\prime$. 

Let $g_S$ denote the Riemannian metric on $S$ constructed from the Riemannian metric $g$ on $M$ and the connection $\alpha$ on the principal bundle $\Pi : S\to M$ in Section \ref{s:inv}. Denote by $f^t$ the geodesic flow of the Riemannian metric $g_S$ on $T^*S$, i.e. the Hamiltonian flow with Hamiltonian $\frac 12|\nu|^2$ on $T^*S$ with the canonical symplectic structure. For any $t\in \mathbb R$ the diffeomorphism $f^t : T^*S\to T^*S$ takes the submanifold $\mu^{-1}(1)$ into itself. The restriction of $f^t$ to $\mu^{-1}(1)$ commutes with the $S^1$-action on $Z$, thus defining a flow on $B=\mu^{-1 }(1)/S^1$. This flow is called the Hamiltonian reduction of the Riemannian geodesic flow $f^t$. It is a Hamiltonian flow on $B$ equipped with the reduced symplectic structure.
The  connection $\alpha$ defines the isometry $d\Pi^*_p : T^*_xM\stackrel{\cong}{\to} \mathcal H^*_p$. Therefore, the Hamiltonian of the reduced flow is $\frac 12(|\xi|^2+1)$, and thus, up to time change, the flow coincides with the magnetic geodesic flow $\Phi$ on $B\cong T^* M$. It is important here that the connection used to construct the diffeomorphism $B\cong T^*M$ coincides with the connection that defines the metric $g_S$ (see also Example \ref{ex:1} below).

\begin{ex}\label{ex:1}
Let us consider the case of an exact magnetic system, i.e. the case when the Hermitian line bundle $(L,h^L)$ is trivial and the Hermitian connection $\nabla^L$ is written as $\nabla^L=d-i \mathbf A$ with some real 1-form $\mathbf A $. 

Then the bundle $\Pi : S\to M$ has the form
\[
S=M\times S^1=\{(x,v)\in M\times \mathbb C: |v|=1\}, \quad \Pi(x,v)=x.
\]
For any $N\in \ZZ$ the space $E_N$ consists of smooth functions on $S$ of the form
\begin{equation}\label{e:iso-exact}
f(x,e^{i\theta})=s(x)e^{iN\theta}, \quad x\in M, \quad \theta \in \mathbb R/2\pi\mathbb Z,
\end{equation}
where $s\in C^\infty(M)\cong C^\infty(M,L^N)$ (cf. \ref{e:iso1}). 

The connection form $\alpha$ on $S$ is given by 
\[
\alpha(x,v)=d\theta+\mathbf A(x), \quad (x,v)\in S.
\]
The corresponding horizontal subspace has the form:
\[
H_{(x,\theta)}=\{V-\langle \mathbf A(x),V\rangle\frac{\partial}{\partial\theta}: V\in T_{(x,\theta) }S\}.
\]
The subspace $V^*_{(x,\theta)}$ is generated by the covector $\alpha(x,v)$. Therefore, the horizontal de Rham differential $d_{\rm h} : C^\infty(S)\to \Omega^1$ is given by
\[
d_{\rm h} f(x,\theta) =df-\frac{\partial f}{\partial\theta}\alpha=d_Xf- \mathbf A(x)\frac{\partial f}{\partial \theta}.
\]
Finally, the horizontal Laplacian $\Delta_{\rm h}$ has the form
\[
\Delta_{\rm h}=\left(d_X-\mathbf A(x)\frac{\partial}{\partial\theta}\right)^* \left(d_X-\mathbf A(x)\frac{ \partial}{\partial\theta}\right).
\]
It is easy to see that its restriction to $E_N$ corresponds 
under the isomorphism \eqref{e:iso-exact} to the operator $(d-iN\mathbf A)^*(d-iN\mathbf A)$.

Recall that the magnetic geodesic flow $\Phi$ coincides with the Hamiltonian reduction of the geodesic flow $f$ on $T^*S$ given by the Riemannian metric $g_S$. Its concrete realization as a flow on $T^*M$ depends on the choice of a connection on $S$. Above, we used for this the same connection $\alpha$ that was used in the definition of the metric $g_S$. But, in principle, we can use another connection, which will lead to a more complex Hamiltonian, not necessarily $1/2(|\xi|^2+1)$.
For example, in the example under consideration of an exact magnetic system, we can take the trivial connection $\alpha^\prime=d\theta$ to construct a diffeomorphism $T^*M\cong B$. It is easy to see that in this case the reduced symplectic manifold coincides with the manifold $T^*M$ endowed with the canonical symplectic structure, and the reduced Hamiltonian flow coincides with the Hamiltonian flow with the Hamiltonian
$\frac{1}{2}(|\xi-\mathbf A(x)|^2_{g^{-1}}+1)$,
i.e. up to reparametrization with the Hamiltonian flow $\phi^t: T^*M\to T^*M$ with the Hamiltonian $H$ given by \eqref{e:H-Hsub}.
These two realizations are related by the diffeomorphism $T(x,\xi)=(x, \xi -\mathbf A(x))$ (cf. \eqref{e:defT}).
\end{ex}

\subsection{The case of non-zero energy}\label{s:GU}
In \cite{Gu-Uribe89}, Guillemin and Uribe considered another version of the smoothed spectral density of the operator $H_N=\Delta^{L^N}$. For $E> 1$ and $\varphi\in \mathcal S(\RR)$ it is given by
\begin{equation}
\label{e:Yp-GU}
\mathcal Y_{N}(\varphi)=\operatorname{tr} \varphi(\sqrt{\Delta^{L^N}+N^2}-EN).
\end{equation}
For exact magnetic systems, this formula can be rewritten in the following form:
\begin{equation}
\label{e:Yp-GU1}
\mathcal Y_N(\varphi)=\operatorname{tr} \varphi\left(\frac{\sqrt{\mathcal H^\hbar+1}-E}{\hbar}\right).
\end{equation}
Comparing the formulas \eqref{e:defYp} and \eqref{e:Yp-GU1}, one can naturally conclude that the parameters $E$ in \eqref{e:Yp-GU} and $E_0$ in \eqref {e:defYp2} are related by $E=\sqrt{E_0+1}$.

Let us briefly describe the main ideas of \cite{Gu-Uribe89}, based on the application of the method described in Section \ref{s:inv}. We will use the notation and constructions described in this section. Denote by $\Delta_S$ the Laplace-Beltrami operator of the Riemannian metric $g_S$. The equality holds true:
\[
\Delta_S =D^2_\theta + \Delta_{\rm h}.
\]
The operator $\Delta_S$ commutes with the operators $D_\theta$ and $\Delta_{\rm h}$.

Consider the operator $P=\Delta_S^{1/2}$, which is a first order elliptic pseudodifferential operator on $S$.
It is easy to see that the operator $\sqrt{\Delta^{L^N}+N^2}$ corresponds under the isomorphism \eqref{e:iso} to the restriction of the operator $P$ to $E_N$.

Consider the distribution $Y\in \mathcal D^\prime(\mathbb R)$ given by 
\begin{equation}\label{e:YE}
Y(s)=\sum_{N=1}^{\infty}\mathcal Y_N(\varphi)e^{iNs},\quad s\in \mathbb R.
\end{equation}
It is a $2\pi$-periodic distribution of the variable $s$ and belongs to the generalized Hardy class, i.e. its Fourier series contains only positive frequencies.

The distribution $Y\in \mathcal D^\prime(\RR)$ can be interpreted as the distributional trace of the operator $\varphi(P-ED_\theta)e^{is D_\theta}$. For any $f\in C^\infty_c(\RR)$ the equality holds true:
\[
\langle Y, f\rangle =\operatorname{tr} \int_{-\infty}^{+\infty}\varphi(P-ED_\theta)e^{is D_\theta}f(s)ds=2 \pi\operatorname{tr} \varphi(P-ED_\theta)\check{f}(A),
\]
where $\check{f}$ denotes the inverse Fourier transform of $f$. Here the operator $P-ED_\theta$ is not necessarily an elliptic operator, so the operator $\varphi(P-ED_\theta)$ is not, generally speaking, a trace class operator. But the operators $P-ED_\theta$ and $D_\theta$ are commuting, jointly elliptic operators on $S$, which allows us to prove that $\varphi(P-ED_\theta)\check{f}(D_\theta)$ is a smoothing operator and therefore its trace is well defined.

In the paper \cite{Gu-Uribe89} the authors carried out an analysis of the distribution $Y$ in the spirit of the proof of the Duistermaat-Guillemin trace formula \cite{Du-Gu75}, which allowed them to prove the existence of an asymptotic expansion as $N\to\infty$ of the sequence $\mathcal Y_N$ with arbitrary $E>1$ and $\varphi\in \mathcal S(\RR)$ with compactly supported Fourier transform under the condition that the magnetic geodesic flow on the corresponding energy level set is clean (see \cite[Corollary 7.2] {Gu-Uribe89}, as well as \cite[Theorem 2.1]{BU91}). We refer the reader to \cite{Gu-Uribe89,KT19} for more details.

\subsection{The case of zero energy}\label{s:BU} In \cite{BU07} D. Borthwick and A. Uribe investigated the structure of low-lying eigenvalues of the operator $H_N$ under the condition that the form $F$ has maximal rank. Their main goal was to prove the asymptotic expansion of the associated generalized Bergman kernels, but nevertheless, the trace formula at the zero energy level for the function $Y_N(\varphi)$ given by \eqref{e:defYp2-cr} follows immediately from the results of this paper.
Let us briefly describe the main ideas of this work. We will use the notation and constructions described in Section \ref{s:inv}.

The operator $H_N$ corresponds under the isomorphism \eqref{e:iso} to the restriction of the operator
\[
P=\Delta_{\rm h}+(\Pi^*V) D_\theta
\]
to the subspace $E_N$. Here $\Pi^*V\in C^\infty(S)$ denotes the lift of $V$ by $S$. (Note that the corresponding multiplication operator commutes with $D_\theta$.)
Therefore, the operator $\frac{1}{N}H_N$ corresponds to the restriction of  the first-order pseudodifferential operator
\[
(D_\theta)^{-1}P=(D_\theta)^{-1}\Delta_{\rm h}+\Pi^*V
\]
to the subspace $E_N$. Some problem is that the last operator, generally speaking, is not a pseudodifferential operator. It has singularities where the operator $D_\theta$ is not elliptic. However, these singularities are outside the characteristic manifold $\mathcal Z $ of the operator $\Delta_{\rm h}$ (see \eqref{e:char1}), which allows us to replace the operator $(D_\theta)^{-1}P$ with another operator $A$, which is microlocally equal to it outside some neighborhood of the characteristic manifold of $D_\theta$ and is a standard pseudodifferential operator with double symplectic characteristics.

Recall that $\Delta_S$ denotes the Laplace-Beltrami operator of the Riemannian metric $g_S$. The operator
\[
\Delta_S +(\Pi^*V) D_\theta=P + D^2_\theta
\]
is a second-order elliptic operator with a positive principal symbol that commutes with $P$. Therefore, there is defined an operator 
\[
F:= \sqrt{P + D^2_\theta},
\]
such that $F^2-(S + D^2_\theta)$ is a smoothing operator of finite rank. The operator $F$ is a standard first-order elliptic pseudodifferential operator that commutes with $P$ and $D_\theta$.

Let $f\in C^\infty(\mathbb R)$ be a non-negative cutoff function identically equal to zero in a neighborhood of zero. The operator $D_\theta^2F^{-2}$ is a classical zero order pseudodifferential operator. Using well-known results on the functional calculus for zero-order pseudodifferential operators, one can show that the operator
\[
Q:=f(D_\theta^2F^{-2})D_\theta^{-1}
\]
is well defined and is a classical pseudodifferential operator of order $-1$. Moreover, the principal symbol of $Q$ is equal to $\sigma(D_\theta)^{-1}$ in some conic neighborhood of $\mathcal Z$.

We define the operator $A$ by 
\[
A := QP = f(D_\theta^2F^{-2})(D_\theta^{-1}\Delta_{\rm h} +(\Pi^*V)).
\]
Then $A$ is a classical first-order pseudodifferential operator on $S$ with double characteristics. Its principal symbol is $\sigma(\Delta_{\rm h})/\sigma(D_\theta)$ in some conic neighborhood of $\mathcal Z$.

Before stating the main result, let us briefly recall some facts about Fourier integral operators of Hermite type that were introduced in \cite{B74}. These operators differ from the standard Fourier integral operators in that they are associated with isotropic rather than Lagrangian submanifolds of the cotangent bundle. The motivation for such a generalization of Fourier integral operators was the microlocal description of the structure of the Szego projector on the boundary of a pseudoconvex domain given in \cite{BMS76}. The calculus of Fourier integral operators of Hermite type is closely related to the calculus of Fourier integral operators with complex phase \cite{Melin-Sj} and Maslov's complex germ method \cite{Maslov}.

Let $\mathcal V$ be a smooth manifold and $C \subset T^*\mathcal V\setminus\{0\}$ be a homogeneous isotropic submanifold. Assume that $C$ is closed, homogeneous (that is, if $(x,\xi)\in C$, then $(x, \lambda\xi)\in C$ for any $\lambda>0$) and isotropic (that is, the restriction of the canonical symplectic form to $C$ vanishes). A non-degenerate phase function is a function $\psi \in C^\infty(\mathcal V \times B, \mathbb R)$, where $B$ is an open conical subset of $(\mathbb R \times\mathbb R^ n)\setminus\{0\}$ with coordinates $(\tau, \eta)$ satisfying the following conditions:

(1) $\psi(x, \tau, \eta)$ is homogeneous in $(\tau, \eta)$.

(2) $d\psi$ never vanishes.

(3) The critical set of $\psi$ given by 
\[
C_\psi = \{(x, \tau, \eta)\in \mathcal V\times B : (d_\tau\psi)(x,\tau, \eta) = (d_\eta\psi)(x ,\tau, \eta) = 0\},
\]
intersects transversally the subspace $\{\eta= 0\}$.

(4) The map
\[
(x, \tau, \eta)\in \mathcal V\times B \mapsto \left(\frac{\partial \psi}{\partial \tau}, \frac{\partial \psi}{\partial \eta_1 }, \ldots, \frac{\partial \psi}{\partial \eta_n}\right) \in \mathbb R^{n+1},
\]
has rank $n+1$ at each point of $C_\psi$.

We define the map $F : C_\psi\to T^*\mathcal V$ by 
\[
F: (x, \tau, \eta) \mapsto (x, (d_x\psi)(x, \tau, \eta)).
\]
The image under the map $F$ of the subspace $\{\eta = 0\}\cap C_\psi$ is a homogeneous isotropic submanifold $\Sigma$ of the manifold $T^*\mathcal V$ of dimension $n+1$. We will say that the phase function $\psi$ parametrizes the submanifold $\Sigma$.
This definition coincides with the standard definition in the case when the submanifold is Lagrangian.

The space $I^m(\mathcal V, \Sigma)$ of Hermitian Fourier distributions associated with a homogeneous isotropic submanifold $\Sigma\subset T^*M$ consists of generalized functions on $\mathcal V$ that are locally representable as oscillating integrals of the form
\[
\int e^{i\psi(x,\tau,\eta)}a\left(x, \tau, \frac{\eta}{\sqrt{\tau}}\right) d\tau\, d \eta,
\]
where the phase $\psi$ parametrizes the submanifold $\Sigma$ and the amplitude $a(x, \tau, u)\in C^\infty(U \times B)$ satisfies the following conditions:

(1) For any $R>0$, any multi-indices $\alpha$, $\beta$, and $\gamma$, and any compact set $K\subset\subset U$, there exists a constant $C>0$ such that
\[
|D^\alpha_xD^\beta_\tau D^\gamma_ua(x, \tau, u)|\leq C|\tau|^{m-|\gamma|}(1+|u|)^{-R },\quad x\in K. \quad (\tau,u)\in B.
\]

(2) $a(x, \tau, u)$ is equal to zero in the neighborhood of $\tau = 0$.

(3) $a(x, \tau, u)$ admits an asymptotic expansion of the form
\[
a(x, \tau, u)\sim \sum_{i=0}^\infty \tau^{m_i}a_i(x, \tau, u),\quad \tau\to +\infty,
\]
where $m_i\in \frac 12 \mathbb Z$, where $m_0 = m -1/2$, $m_i$ is strictly decreasing, and $m_i \to-\infty$.

The conditions on the phase function guarantee that the wavefront of any Hermitian Fourier distribution of class $I^m(\mathcal V, \Sigma)$ is contained in $\Sigma$. The principal symbol of a Hermitian Fourier distribution from the class $I^m(\mathcal V, \Sigma)$ is a symplectic spinor, which is a half-density along $\Sigma$ tensored by a smooth vector in the metaplectic representation associated with the symplectic normal bundle to $ \Sigma$ \cite{Gu75,BG81}.

In our case, consider the set $\mathcal Z^\Delta \subset T ^*(S\times S)\cong T^*S\times T^*S$ given by (cf. \eqref{e:char1})
\[
\mathcal Z^\Delta =\{(p,r\alpha_p, p, -r\alpha_p)\subset T ^*(S\times S) : p\in S, r>0\}.
\]
It is easy to check that $\mathcal Z^\Delta$ is a homogeneous isotropic submanifold in $T ^*(S\times S)$.

\begin{thm}[\cite{BU07}]
Let $\varphi$ be a function from $\mathcal S(\mathbb R)$ whose Fourier transform is compactly supported. Then the operator $\varphi(A)$ is a Integral Fourier operator of Hermite type with Schwartz kernel belonging to the class $I^{1/2}(S\times S, \mathcal Z^\Delta)$.
\end{thm}

This theorem immediately implies the existence of an asymptotic expansion \eqref{e:main-tr} of Theorem \ref{t:trace} for any function $\varphi \in \mathcal S(\mathbb R)$ whose Fourier transform is compactly supported, in the case when $F$ has maximal rank (see \cite{BU07} for more details). 

\begin{rem}\label{r:E0}
For an arbitrary $E_0>0$, the operator $\frac{1}{N}H_N-E_0N$ corresponds to the restriction of the first-order pseudodifferential operator 
\[
(D_\theta)^{-1}\Delta_{\rm h}+\Pi^*V -E_0D_\theta
\]
to the subspace $E_N$. It is quite possible that by applying the methods of \cite{PU95,BU07} to this operator, one can prove an asymptotic formula for the function $Y_N(\varphi)$ given by \eqref{e:defYp2} for $E_0>0$. We will discuss these issues elsewhere. 
\end{rem}

\section{Examples} \label{s:examples}

In this section, we give concrete examples of calculating the trace formula at zero energy level for two-dimensional surfaces of constant curvature with constant magnetic fields. These examples have already been considered in papers \cite{KT19,KT22} in the case of non-zero energy. Therefore, we will be brief in their description, referring the reader to \cite{KT19,KT22} for more details. We will also consider an example of constant magnetic field on a three-dimensional torus as the simplest example of a non-maximal rank magnetic system. 

\subsection{Constant magnetic field on a two-dimensional torus}\label{s:torus}
Consider the two-dimensional torus $\mathbb T^2=\mathbb R^2/\mathbb Z^2$, endowed with the standard flat Riemannian metric
\[
g=dx^2+dy^2.
\]
The magnetic field $F$ is given by 
\[
F=2\pi dx\wedge dy.
\]

The corresponding line bundle $L$ on $\mathbb T^2$ consists of the equivalence classes of triples $(x,y,u)\in \mathbb R^2\times \mathbb C$, where
\[
(x+1,y,u)\sim (x,y,e^{-2\pi iy}u), \quad (x,y+1,u)\sim (x,y,u)
\]
with the projection $L\to \mathbb T^2$ given by $(x,y,u)\in L\mapsto (x,y)\in \mathbb T^2$. The space of its smooth sections is identified with the space of functions $u\in C^\infty(\RR^2)$ such that
\begin{equation}\label{e:torus-sections}
u(x+1,y)=e^{2\pi iy}u(x,y),\quad u(x,y+1)=u(x,y), \quad (x,y)\in \mathbb R^2.
\end{equation}
The Hermitian connection on $L$ is defined by 
\[
\nabla^L=d - 2\pi i x\,dy.
\]
Consider the operator
\[
H_N =\Delta^{L^N}=-\frac{\partial^2}{\partial x^2}-\left(\frac{\partial}{\partial y}-2\pi Ni x\right )^2.
\]

\begin{thm}\label{t:5}
For any $\varphi\in \mathcal S(\RR)$, the smoothed spectral density $Y_N(\varphi)$ of the operator $H_N$ given by \eqref{e:defYp2-cr} has the form
\begin{equation}\label{e:2.1}
Y_N(\varphi)= f_0(\varphi)N,
\end{equation}
where (cf. \eqref{e:KD1a})
\begin{equation}
f_0(\varphi)=\frac{i}{4\pi}\left\langle \frac{1}{\sin 2\pi(t+i0)}, \hat\varphi(t)\right\rangle.\label{e:c0-torus}
\end{equation}
\end{thm}

\begin{proof}
The operator $\Delta^{L^N}$ has eigenvalues of the form
\[
\nu_{N,j}=2\pi N(2j+1),\quad j=0,1,2,\ldots,
\]
with multiplicity
\[
m_{N,j}=N.
\]
Therefore, the function $Y_N(\varphi)$ is given by 
\[
Y_N(\varphi)=\sum_{j=0}^{\infty}N\varphi(2\pi (2j+1))= f_0(\varphi)N,
\]
where
\[
f_0(\varphi)=\sum_{j=0}^{\infty}\varphi(2\pi (2j+1)).
\]
It remains to use the formula \eqref{e:KD1}.\end{proof}

\subsection{Constant magnetic field on a three-dimensional torus}\label{s:torus3}
Consider the three-dimensional torus $\mathbb T^3=\mathbb R^3/\mathbb Z^3$,
endowed with the standard flat Riemannian metric
\[
g=dx^2+dy^2+dz^2.
\]
Let us assume that the form $F$ is given by
\[
F=2\pi dx\wedge dy.
\]
For the corresponding line bundle $L$ on $\mathbb T^3$, the space of its smooth sections is identified with the space of $u\in C^\infty(\RR^3)$ such that for any $(x,y,z )\in \mathbb R^3$
\[
u(x+1,y,z)=e^{2\pi iy}u(x,y,z),\quad u(x,y+1,z)= u(x,y,z+1 )=u(x,y,z).
\]
The Hermitian connection on $L$ is defined by 
\[
\nabla^L=d - 2\pi i x\,dy.
\]
Consider the operator
\[
H_N =\Delta^{L^N}=-\frac{\partial^2}{\partial x^2}-\left(\frac{\partial}{\partial y}-2\pi Ni x\right )^2-\frac{\partial^2}{\partial z^2}.
\]

\begin{thm}
For any $\varphi\in \mathcal S(\RR)$, the smoothed spectral density $Y_N(\varphi)$ of the operator $H_N$ given by \eqref{e:defYp2-cr} has the form
\[
Y_N(\varphi)= f_0(\varphi)N^{3/2}+\mathcal O(N^{-\infty}),
\]
where (cf. \eqref{e:KD1b})
\[
f_0(\varphi)=\frac{e^{-(1/4)\pi i}}{8\pi^{3/2}}\left\langle \frac{1}{(t+i0)^ {1/2}\sin 2\pi(t+i0)}, \hat\varphi(t)\right\rangle.
\]
\end{thm}

\begin{proof}
The eigenvalues of $H_N$ are computed using the separation of variables:
\[
\nu_{N,j, k}=2\pi N(2j+1)+(2\pi k)^2,\quad j=0,1,2,\ldots,\quad k\in \mathbb Z ,
\]
with multiplicity
\[
m_{N,j,k}=N.
\]
Therefore, the function $Y_N(\varphi)$ has the form
\[
Y_N(\varphi)=\sum_{j=0}^{\infty}\sum_{k\in \mathbb Z} N\varphi(2\pi (2j+1)+\frac{1}{N}( 2\pi k)^2).
\]
We write this formula as
\[
Y_N(\varphi)=\frac 12N\sum_{k\in \mathbb Z}f(\frac{k}{\sqrt{N}}),
\]
where $f\in \mathcal S(\mathbb R)$ is given by
\[
f(x)=\sum_{j=0}^{\infty} \varphi(2\pi (2j+1)+(2\pi x)^2),
\]
and apply the Poisson summation formula. We get
\[
Y_N(\varphi)=\frac 12N^{3/2}\sum_{m\in \ZZ}\hat{f}(2\pi \sqrt{N} m).
\]
Since $\hat{f}(2\pi \sqrt{N} m)=\mathcal O(N^{-\infty})$ for $m\neq 0$, we conclude that
\[
Y_N(\varphi)=\frac 12N^{3/2}\hat{f}(0)+ \mathcal O(N^{-\infty}).
\]
It remains to compute $\hat{f}(0)$ using \eqref{e:HDformula-bis}
\begin{align*}
\hat{f}(0)=\int_{-\infty}^{\infty}f(x)\,dx = & \sum_{j=0}^{\infty}\int_{-\infty}^ \infty \varphi(2\pi (2j+1)+(2\pi x)^2)dx\\ =& \frac{1}{2\pi}\sum_{j=0}^{\infty} \int_{0}^\infty \varphi(2\pi (2j+1)+\xi^2)d\xi\\ =& \frac{e^{-(1/4)\pi i}}{ 4\pi^{3/2}}\left\langle\frac{1}{(t+i0)^{1/2}\sin 2\pi (t+i0)}, \hat\varphi \right\rangle,
\end{align*}
which completes the proof.
\end{proof}

\subsection{Two-dimensional sphere}
Consider the two-dimensional sphere
\[
S^2=\{(x,y,z)\in \mathbb R^3 : x^2+y^2+z^2=R^2\},
\]
endowed with the Riemannian metric $g$ induced by the embedding in the Euclidean space $\RR^3$. In spherical coordinates
\[
x=R\sin\theta \cos\varphi, \quad y=R\sin\theta \sin\varphi, \quad z=R\cos\theta, \quad \theta\in (0,\pi), \varphi\in(0,2\pi),
\]
the metric $g$ has the form
\[
g=R^2(d\theta^2+\sin^2\theta d\varphi^2).
\]
Let us assume that the form $F$ is given by
\[
F=\frac 12\sin\theta d\theta\wedge d\varphi.
\]
The corresponding line bundle $L$ is the line bundle associated with the Hopf bundle $S^3\to S^2$ and the character $\chi : S^1\to S^1$, $\chi(u)=u, u \in S^1$.

\begin{thm}\label{t:6}
For any $\varphi\in \mathcal S(\RR)$, the smoothed spectral density $Y_N(\varphi)$ of $H_N=\Delta^{L^N}$ given by \eqref{e:defYp2-cr} has an asymptotic expansion
\begin{equation}\label{e:3.1}
Y_N(\varphi) \sim \sum_{j=0}^\infty f_j(\varphi)N^{1-j},\quad N\to \infty.
\end{equation}
The coefficients $f_j$ are computed explicitly. For the first two of them we have
\begin{equation}\label{e:c0-sphere}
f_0(\varphi) =\frac{i}{4\pi}\left\langle \frac{1}{\sin \frac{1}{2R^2}(t+i0)}, \hat\varphi( t)\right\rangle,
\end{equation}
\begin{equation}\label{e:c1-sphere}
f_1(\varphi) =\frac{1}{4\pi}\left\langle \left[R^2\frac{d^2}{dt^2}+\frac{1}{4R^{2} }\right]\frac{t}{\sin \frac{1}{2R^2}(t+i0)}, \hat\varphi(t) \right\rangle.
\end{equation}
\end{thm}

\begin{proof}
The spectrum of $\Delta^{L^N}$ consists of the eigenvalues
\[
\nu_{N,j}=\frac{1}{R^2}\left[j(j+1)+\frac{N}{2}(2j+1)\right], \quad j=0 ,1,2, \ldots,
\]
with multiplicity
\[
m_{N,j}=N+2j+1.
\]
Therefore, the function $Y_N(\varphi)$ has the form:
\begin{equation} \label{e:sphere-Yp}
Y_N(\varphi)=\sum_{j=0}^\infty(N+2j+1)\varphi\left(\frac{1}{R^2}\left[\frac{1}{N}j (j+1)+\frac{1}{2}(2j+1)\right]\right).
\end{equation}
The Taylor series expansion gives us the asymptotic expansion
\begin{multline*}
\varphi\left(\frac{1}{R^2}\left[\frac{1}{N}j(j+1)+\frac{1}{2}(2j+1)\right]\right)\\ \sim \sum_{k=0}^{\infty}\frac{1}{k!}\frac{1}{R^{2k}}\frac{1}{N^k} \varphi^{(k)}\left(\frac{1}{2R^2}(2j+1)\right)j^k(j+1)^k.
\end{multline*}
Substituting it into \eqref{e:sphere-Yp} proves
the existence of the asymptotic expansion \eqref{e:3.1}.

Its leading coefficient is given by  
\[
f_0(\varphi) = \sum_{j=0}^\infty\varphi\left(\frac{1}{2R^2}(2j+1)\right)
\]
whence using \eqref{e:KD1} we get \eqref{e:c0-sphere}.

For the next coefficient $c_1(\varphi)$ we have
\[
f_1(\varphi) =\frac{1}{R^{2}}\sum_{j=0}^\infty \varphi^{\prime}\left(\frac{1}{2R^2}(2j +1)\right)j(j+1)+\sum_{j=0}^\infty \varphi\left(\frac{1}{2R^2}(2j+1)\right)(2j+1 ),
\]
whence we obtain \eqref{e:c1-sphere} using \eqref{e:KD1} and properties of the Fourier transform.
\end{proof} 

As an illustration, we compute the linearized flow.
This magnetic system is exact in spherical coordinates, and the principal symbol given by \eqref{e:H-Hsub} has the form
\[
H(\theta,\varphi,p_\theta,p_\varphi)=\frac{1}{R^2}p^2_\theta+\frac{1}{R^2\sin^2\theta}\left (p_\varphi+\frac 12\cos\theta\right)^2, \quad (\theta,\varphi,p_\theta,p_\varphi)\in T^*S^2.
\]
The characteristic manifold $X_0$ is given by   (cf. \eqref{e:Sigma-exact})
\[
p_\varphi+\frac 12\cos\theta=0, \quad p_\theta=0.
\]
The spherical coordinates $(\theta,\varphi)$ determine the coordinates on $X_0$ by the map
\[
j(\theta,\varphi)=(\theta,\varphi,0,-\frac 12\cos\theta)\in X_0.
\]
The Hamiltonian flow $\phi^t$ with the Hamiltonian $H$ is given by the system of equations
\begin{equation}\label{e:lin-var}
\begin{aligned}
\dot{\theta}=&\frac{2}{R^2}p_\theta, \quad \dot{\varphi}=\frac{2}{R^2\sin^2\theta}\left( p_\varphi+\frac 12\cos\theta\right),\\
\dot{p}_\theta=&\frac{2\cos\theta}{R^2\sin^3\theta}\left(p_\varphi+\frac 12\cos\theta\right)^2+\frac{1}{R^2\sin\theta}\left(p_\varphi+\frac 12\cos\theta\right), \quad \dot{p}_\varphi=0.
\end{aligned}
\end{equation}
Any point $j(\theta_0,\varphi_0)\in X_0$ is a fixed point of the flow, and thus defines a constant solution of the Hamiltonian system \eqref{e:lin-var}:
\[
\theta(t)=\theta_0,\quad \varphi(t)=\varphi_0,\quad p_\theta(t)=0, \quad p_\varphi(t)=-\frac 12\cos\theta_0.
\]
Computing the system of variational equations for the system \eqref{e:lin-var} along this solution, we obtain a system of first-order differential equations defining the flow $d\phi_{t,j(\theta_0,\varphi_0)}$ on $T_{j (\theta_0,\varphi_0)}(T^*S^2)$:
\begin{equation}\label{e:dphit}
\begin{aligned}
\dot{\Theta}=& \frac{2}{R^2}P_\theta, \quad \dot{\Phi}=\frac{2}{R^2\sin^2\theta_0}\left( P_\varphi-\frac 12\sin\theta_0 \Theta\right),\\
\dot{P}_\theta=& \frac{1}{R^2\sin\theta_0}\left(P_\varphi-\frac 12\sin\theta_0 \Theta\right), \quad \dot{P }_\varphi=0.
\end{aligned}
\end{equation}
The tangent space $T_{j(\theta_0,\varphi_0)}X_0$ is given by  
\[
P_\theta=0, \quad P_\varphi-\frac 12\sin\theta_0 \Theta=0,
\]
therefore, as linear coordinates on $N_{j(\theta_0,\varphi_0)}X_0$ we can take
\[
\hat P_\theta=P_\theta, \quad \hat P_\varphi=P_\varphi-\frac 12\sin\theta_0 \Theta=0.
\]
From \eqref{e:dphit} we obtain a system of first-order differential equations defining the linearized flow $d\phi_{t,j(\theta_0,\varphi_0)}$ on $N_{j(\theta_0,\varphi_0)}X_0$:
\[
\dot{\hat{P_\theta}}=\frac{1}{R^2\sin\theta_0}\hat P_\varphi. \quad \dot{\hat{P_\varphi}}=-\frac{1}{R^2}\sin\theta_0\hat P_\theta.
\]
Thus, the matrix of the linear map $d\phi_{t,j(\theta_0,\varphi_0)}$ of the space $N_{j(\theta_0,\varphi_0)}X_0$ in the coordinates $(\hat P_\theta, \hat P_\varphi)$ has the form (cf. \eqref{e:lin-flow})
\[
d\phi_{t,j(\theta_0,\varphi_0)}=
\begin{pmatrix}
\cos \frac{1}{R^2}t & \frac{1}{\sin\theta_0}\sin \frac{1}{R^2}t\\
- \sin\theta_0 \sin\frac{1}{R^2}t &\cos \frac{1}{R^2}t
\end{pmatrix}.
\]

\subsection{Hyperbolic plane}
Consider the hyperbolic plane $\mathbb H=\{(x,y)\in \RR^2 : y>0\}$ endowed with the Riemannian metric
\[
g=\frac{R^2}{y^2}(dx^2+dy^2).
\]
Let $\Gamma\subset PSL(2,\mathbb R)$ be a cocompact lattice acting freely on $\mathbb H$ and $M=\Gamma\setminus\mathbb H$ be the corresponding Riemann surface.

We define the form $F$ on $M$ so that its lift to $\mathbb H$ has the form
\[
\tilde F=\frac{dx\wedge dy}{y^2}.
\]
Sections of the Hermitian bundle $L$ on $M$ are identified with functions $\psi$ on $\mathbb H$ satisfying the condition
\[
\psi(\gamma z)=\psi(z)\exp(-2i\arg(cz+d))=\left(\frac{cz+d}{|cz+d|}\right)^{- 2}\psi(z)
\]
for any $z=x+iy\in \mathbb H$ and $\gamma=\begin{pmatrix} a & b\\ c & d\end{pmatrix}\in \Gamma$.

Define the connection form by 
\[
A=\frac{1}{y}dx.
\]

\begin{thm}\label{t:hyper}
For any function $\varphi\in \mathcal S(\RR)$, the smoothed spectral density $Y_N(\varphi)$ of the operator $H_N=\Delta^{L^N}$ given by \eqref{e:defYp2-cr} has an asymptotic expansion
\begin{equation}\label{e:4.1}
Y_N(\varphi) \sim \sum_{j=0}^\infty c_j(\varphi)N^{1-j}, \quad N\to \infty,
\end{equation}
The coefficients $c_j$ are computed explicitly. For the first two of them we have
\begin{equation}\label{e:c0-hyper}
f_0(\varphi) =\frac{i}{2\pi}(g-1)\left\langle \frac{1}{\sin \frac{1}{2R^2}(t+i0)}, \hat\varphi(t)\right\rangle.
\end{equation}
\begin{multline}\label{e:c1-hyper}
f_1(\varphi) =\\ -\frac{1}{2\pi}(g-1)\left\langle \left(\left[R^2\frac{d^2}{dt^2}+ \frac{1}{4R^{2}}\right]t-R^2\frac{d}{dt}\right)\frac{1}{\sin \frac{1}{2R^2}(t+ i0)} , \hat\varphi(t) \right\rangle.
\end{multline}
\end{thm}

\begin{proof}
For any $N\in \NN$ the spectrum of $\Delta^{L^N}$ consists of two parts. Its spectrum on the interval $[0,N^2/R^2]$ consists of the eigenvalues
\[
\nu^{(i)}_{N,j} =\frac{1}{R^2}\left((2j+1)N-j(j+1)\right), \quad 0\leq j
\leq N-1
\]
with multiplicity
\[
m_{N,j}=(g-1)(2N-2j-1), \quad 0\leq j\leq N- 1.
\]
Let $\Delta_M$ be the Laplace-Beltrami operator on $(M,g)$. Denote by $\lambda_\ell, \ell=0,1,2,\ldots.$ its eigenvalues taking multiplicities into account. Then the eigenvalues of $\Delta^{L^N}$ on the half-line $(N^2/R^2, \infty)$ are given by 
\[
\nu^{(c)}_{N,\ell} = \lambda_\ell+\frac{1}{R^2}N^2, \quad \ell=0,1,2,\ldots.
\]

We get the following expression for $Y_N(\varphi)$:
\begin{multline} \label{e:hyper-Yp}
Y_N(\varphi)=\sum_{j=0}^{N-1}(g-1)(2N-2j-1) \times \\ \times \varphi\left(\frac{1}{R^ 2}\left((2j+1)-\frac{1}{N}j(j+1)\right)\right)+\sum_{k=1}^{\infty}\varphi\left(\frac{1}{R^2}N+\frac 1N\lambda_k\right).
\end{multline}
It is clear that
\[
\sum_{k=1}^{\infty}\varphi\left(\frac{1}{R^2}N+\frac 1N\lambda_k\right)=\mathcal O(N^{-\infty}), \quad N\to \infty.
\]
Using the Taylor series expansion as above, we obtain the existence of an asymptotic expansion \eqref{e:4.1} and formulas for the leading coefficient in this expansion
\[
f_0(\varphi) =(2g-2)\sum_{j=0}^{N-1}\varphi\left(\frac{1}{2R^2}(2j+1)\right)
\]
and the following one:
\begin{multline*}
f_1(\varphi)=-2(g-1)\frac{1}{R^{2}}\sum_{j=0}^{\infty} \varphi^{\prime}\left(\frac{ 1}{2R^2}(2j+1)\right)j(j+1)\\
-(g-1)\sum_{j=0}^{\infty} \varphi\left(\frac{1}{2R^2}(2j+1)\right)(2j+1),
\end{multline*}
whence, as above, we obtain \eqref{e:c0-hyper} and \eqref{e:c1-hyper} using \eqref{e:KD1} and properties of the Fourier transform.
\end{proof}

\end{document}